\documentclass[twoside,10pt,a4paper]{amsart}

\usepackage{amsthm}

\usepackage{graphics}
\usepackage{multirow}

\usepackage[all]{xy}

\DeclareMathOperator{\im}{im}
\DeclareMathOperator{\Id}{Id}
\DeclareMathOperator{\Hom}{Hom}
\DeclareMathOperator{\Ext}{Ext}
\newcommand{\norm}{\mathtt{N}}
\newcommand{\smatrix}[1]{\left(\!\begin{smallmatrix} #1\end{smallmatrix}\!\right)}
\newcommand{\transpose}[1]{#1^{\! T}}
\newcommand{\Tate}{\widehat{\Ext}}
\newcommand{\class}{\mathcal{C}}

\newtheorem{lemma}[equation]{Proposition}
\newtheorem{definition}[equation]{Definition}
\newtheorem{satz}[equation]{Theorem}
\newtheorem{korollar}[equation]{Corollary}

\theoremstyle{remark}
\newtheorem{bemerkung}[equation]{Remark}
\DeclareMathOperator{\tr}{tr}
\DeclareMathOperator{\coker}{coker}
\DeclareMathOperator{\Mat}{Mat}
\newcommand{\id}{\textrm{id}}

\newcommand{\beweis}[1][]{\textit{Proof#1:}\\[2pt]}

\newcommand{\idx}[1]{{#1}}  
\newcommand{\adx}[1]{{#1}}  
\newcommand{\vdx}[1]{{(#1)}}
\newcommand{\hochschild}{H\! H}

\newcommand{\svec}{\mathbf}

\newcommand{\rem}[1]{}

\numberwithin{equation}{section}
\begin{document}
\SelectTips{cm}{} 

\title[Secondary multiplication in Tate Cohomology of certain $p$-groups]{Secondary multiplication in Tate Cohomology\\ of certain $p$-groups}
\author{Martin Langer}

\begin{abstract}
Let $k$ be a field and let $G$ be a finite group. By a theorem of D.~Benson, H.~Krause and S.~Schwede, there is a canonical element in the Hochschild cohomology of the Tate cohomology $\gamma_G\in \hochschild^{3,-1} \hat{H}^*(G)$ with the following property: Given any graded  $\hat{H}^*(G)$-module $X$, the image of $\gamma_G$ in $\Ext^{3,-1}_{\hat{H}^*(G)} (X,X)$ is zero if and only if $X$ is isomorphic to a direct summand of $\smash{\hat{H}^*(G,M)}$ for some $kG$-module $M$.

Suppose that the characteristic of $k$ is $p\neq 3$. We show that $\gamma_G$ is trivial if $G$ is a (finite) abelian $p$-group of $p$-rank at least $3$. Furthermore, $\gamma_G$ is non-trivial if $G$ is an abelian $p$-group of $p$-rank $2$, or if $p=2$ and $G$ is the quaternion group with $8$ elements. 
\end{abstract}
\maketitle

\section{Introduction}
The starting point of this paper is the following theorem of D.~Benson, H.~Krause and S.~Schwede:
\begin{satz}\label{bkstheorem} \cite{bks} 
Let $k$ be a field, $G$ a finite group, and let $\hat{H}^*(G)$ denote the graded Tate cohomology algebra of $G$ over $k$. Then there exists a canonical element in Hochschild cohomology of $\hat{H}^*(G)$
\[ \gamma_G\in \hochschild^{3,-1} \hat{H}^*(G), \]
such that for any graded $\hat{H}^*(G)$-module $X$, the following are equivalent:
\begin{itemize}
\item[(i)] The image of $\gamma_G$ in $\Ext^{3,-1}_{\hat{H}^*(G)} (X,X)$ is zero.
\item[(ii)] The exists a $kG$-module $M$ such that $X$ is a direct summand of the graded $\hat{H}^*(G)$-module $\hat{H}^*(G,M)$.
\end{itemize}
\end{satz}

In this paper we will compute $\gamma_G$ explicitly for several groups $G$. The plan is as follows: In the first section we will briefly recall the definitions needed in Theorem~\ref{bkstheorem}; most of this part is taken from \cite{bks}, and the reader interested in details should consult that source. In the second section we turn to a first example, namely $G=Q_8$, the quaternion group with $8$ elements. The last two sections are devoted to the case of (finite) abelian $p$-groups.

\subsection*{Acknowledgments} This paper is part of a Diploma thesis written at the Mathematical Institute, University of Bonn. I would like to thank my advisor Stefan Schwede for suggesting the subject and all the helpful comments on the project.

\subsection{Notations and conventions}
All occurring modules will be right modules. We shall often work over a fixed ground field $k$; then $\otimes$ means tensor product over $k$. Whenever convenient, we write $(a_1,a_2,\dots,a_n)$ instead of $a_1\otimes a_2\otimes\dots\otimes a_n$. If $G$ is a group, then $k$ is often considered as a trivial $kG$-module. When no confusion can arise, $k$ will be used as an index variable at the same time.

Let $R$ be a ring with unit, and let $M$ be a $\mathbb{Z}$-graded $R$-module. The degree of every (homogeneous) element $m\in M$ will be denoted by $|m|$. For every integer $n$ the module $M[n]$ is defined by $M[n]^j=M^{n+j}$ for all $j$. Given two such modules $M$ and $L$, a morphism $f:L\longrightarrow M$ is a family $f^j:L^j\longrightarrow M^j$ of $R$-module homomorphisms. The group of all these morphisms is denoted by $\Hom_R(L,M)$. Furthermore, we have $\Hom_R^m (L,M) = \Hom_R(L,M[m])$, the morphisms of degree $m$. The graded module $L\otimes M$ is given by $(L\otimes M)^m =\bigoplus_{i+j=m} L^i\otimes M^j$. If $M$ is a differential graded $R$-module with differential $d$, then the differential of $M[n]$ is given by $(-1)^n d$.

\subsection{Tate Cohomology}\label{tatekohomologie}
Let us recall briefly the definition and basic properties of Tate cohomology. Let $k$ be a field, and let $G$ be a finite group. Then $L=kG$ is a self-injective algebra (i.e.~the classes of projective and injective right-$L$-modules coincide). For any $L$-module $N$ we get a complete projective resolution $\hat{P}_*$ of $N$ by splicing together a projective and an injective resolution of $N$:
\[
\xymatrix@!@=2pt{
\dots & &
\hat{P}_{-2} \ar[ll] &&
\hat{P}_{-1} \ar[ll] &&
\hat{P}_{0} \ar[ll]\ar@{->>}[dl]  &&
\hat{P}_{1} \ar[ll]  &&
\hat{P}_{2} \ar[ll]  &&
\dots  \ar[ll]  \\
&&& & & N\ar@{_{(}->}[ul]
} \]
Given another $L$-module $M$, we can apply the functor $\Hom_L(\_,M)$ to $\hat{P}_*$; then Tate cohomology is defined to be the cohomology groups of the resulting complex:
\begin{equation}\label{tatedef} \Tate_L^n(N,M)=H^n(\Hom_L(\hat{P}_*,M)) \quad\text{for all $n\in\mathbb{Z}$.}  \end{equation}
For arbitrary $L$-modules $X,Y$ and $Z$, we have a cup product
\begin{equation} \label{cupprodukt} \Tate_L^m(Y,Z)\otimes \Tate_L^n(X,Y) \longrightarrow \Tate_L^{m+n}(X,Z), \end{equation}
see e.g.~\cite{Carlson}, §~6. Therefore, $\hat{H}^*(G)=\hat{H}^*(G,k)=\Tate_{kG}^*(k,k)$ is a graded algebra, and $\hat{H}^*(G,M)=\Tate_{kG}^*(k,M)$ is a graded $\hat{H}^*(G)$-module for every $kG$-module $M$. We call a graded $\hat{H}^*(G)$-module $X$ \emph{realisable} if it is isomorphic to $\hat{H}^*(G,M)$ for some $kG$-module $M$.

There is another way of describing the product of $\hat{H}^*(G)$, in terms of $\hat{P}_*$. Consider the differential graded algebra $A=\Hom_L^*(\hat{P}_*,\hat{P}_*)$, which (in degree $n$) is given by
\[ A^n=\prod_{j\in\mathbb{Z}} \Hom_L(\hat{P}_{j+n},\hat{P}_j), \]
and the differential $d:A^n\longrightarrow A^{n+1}$ is defined to be
\[ (df)_j=\partial\circ f_{j+1}-(-1)^n f_j\circ \partial. \]
Here $\partial$ denotes the differential of $\hat{P}_*$. With this definition, the cocycles of $A$ (of degree $n$) are exactly the chain transformations $\hat{P}[n]\rightarrow\hat{P}$, and two cocycles differ by a coboundary if and only if they are chain homotopic. Using standard arguments from homological algebra, one shows that the following map is an isomorphism of $k$-vector spaces:
\begin{equation} \label{tateiso}
\begin{split}
 H^nA &\stackrel{\cong}{\longrightarrow} \Tate^n_L(k,k) \\
 \left[f\right] &\;\mapsto\;  [ \epsilon\circ f_0]
\end{split}
\end{equation}
Here $\epsilon:P_0\longrightarrow k$ is the augmentation. This isomorphism is compatible with the multiplicative structures.

\subsection{Hochschild Cohomology}
We now give a short review of Hochschild cohomology. Let $\Lambda$ be a graded algebra over the field $k$, and let $M$ be a graded $\Lambda$-$\Lambda$-bimodule, the elements of $k$ acting symmetrically. Define a cochain complex $C^{\bullet,*}(\Lambda,M)$ by
\[ C^{n,m}(\Lambda,M)=\Hom_k^m(\Lambda^{\otimes n},M), \]
with a differential $\delta$ of bidegree $(1,0)$ given by
\begin{multline*} (\delta\varphi)(\lambda_1,\dots,\lambda_{n+1})=(-1)^{m|\lambda_1|}\lambda_1\varphi(\lambda_2,\dots,\lambda_{n+1}) \\
+\sum_{i=1}^n (-1)^i\varphi(\lambda_1,\dots,\lambda_i\lambda_{i+1},\dots,\lambda_{n+1}) +(-1)^{n+1}\varphi(\lambda_1,\dots,\lambda_n)\lambda_{n+1}.
\end{multline*}
The Hochschild cohomology groups $\hochschild^{*,*}(\Lambda,M)$ are defined as the cohomology groups of that complex:
\[ \hochschild^{s,t}(\Lambda,M)=H^s(C^{*,t}(\Lambda,M)). \]
In particular, we can regard $M=\Lambda$ as a bimodule over itself; then one writes $\hochschild^{s,t}(\Lambda)=\hochschild^{s,t}(\Lambda,\Lambda)$. For example, an element of $\hochschild^{3,-1}(\Lambda)$ is represented by a family of $k$-linear maps
\[ m=\{m_{i,j,l}:\Lambda^i\otimes\Lambda^j\otimes \Lambda^l\longrightarrow
        \Lambda^{i+j+l-1}\}_{i,j,l\in\mathbb{Z}} \]
satisfying the cocycle relation
\[ (-1)^{|a|}a\cdot m(b,c,d)-m(ab,c,d)+m(a,bc,d)-m(a,b,cd)+m(a,b,c)\cdot d=0 \]
for all $a,b,c,d\in\Lambda$.

Whenever $X$ and $Y$ are $\Lambda$-$\Lambda$-bimodules, one has a cup product pairing
\[ \cup:\Hom_\Lambda(X,Y)\otimes \hochschild^{*,*}\Lambda\longrightarrow \Ext^{*,*}_\Lambda(X,Y). \]
Here $\Ext^{s,t}_\Lambda(X,Y)$ is defined to be $\Ext^{s}_\Lambda(X,Y[t])$. In particular, we have the map
\begin{eqnarray*}
 \hochschild^{3,-1} \hat{H}^*(G) &\longrightarrow& \Ext^{3,-1}_{\hat{H}^*(G)} (X,X) \\
   \phi & \mapsto & \Id_X\cup\,\phi
\end{eqnarray*}
for every $\hat{H}^*(G)$-module $X$. This is the map occurring in the statement of Theorem~\ref{bkstheorem}.

\subsection{The canonical element $\gamma$} \label{kanonischel}
We are now going to describe the construction of the element $\gamma$ mentioned in Theorem~\ref{bkstheorem}. More generally, we will construct an element $\gamma_A\in\hochschild^{3,-1} H^*A$ for every differential graded algebra $A$ over $k$; then we can take $A$ to be the endomorphism algebra of a complete projective resolution of $k$ as a trivial $kG$-module to get $\gamma_G\in \hochschild^{3,-1} \hat{H}^*(G)$.

For a dg-algebra $A$ consider $H^*A$ as a differential graded $k$-module with trivial differential. Then choose a morphism of dg-$k$-modules $f_1:H^*A\longrightarrow A$ of degree $0$ which induces the identity in cohomology. This is the same as choosing a representative in $A$ for every class in $H^*A$ in a $k$-linear way. For every two elements $x,y\in H^*A$, $f_1(xy)-f_1(x)f_1(y)$ is null-homotopic; therefore, we can choose a morphism of graded modules
\[ f_2:H^*A\otimes H^*A\longrightarrow A \]
of degree $-1$ such that for all $x,y\in H^*A$ we have
\[ df_2(x,y)=f_1(xy)-f_1(x)f_1(y). \]
Then for all $a,b,c\in H^*A$,
\begin{equation} \label{mkonstruktion}
 f_2(a,b)f_1(c)-f_2(a,bc)+f_2(ab,c)-(-1)^{|a|}f_1(a)f_2(b,c) 
\end{equation}
is a cocycle in $A$, the cohomology class of which will be denoted by $m(a,b,c)$. This defines a map $m:(H^*A)^{\otimes 3}\longrightarrow H^*A$ of degree $-1$. An explicit computation shows that $m$ is a Hochschild cocycle, thereby representing a class $\gamma_A\in\hochschild^{3,-1} H^*A$. This class is independent of the choices made. 

\subsection{Cyclic $p$-groups}  \label{szyklischegruppen}
Let us go through an example, namely that of cyclic $p$-groups. In fact, this is a special case of an example worked out in \cite{bks}, §~7. Since we will need parts of the computation later on, we recall it here in detail. Denote by $C_n=\mathbb{Z}/n\mathbb{Z}$ the cyclic group of order $n$.

\begin{satz} \label{zyklsatz} Let $k$ be a field of characteristic $p$ and $n=p^\nu$. The Tate cohomology of $C_2$ is a Laurent polynomial ring on a $1$-dimensional class $x$,
\[ \hat{H}^*(C_2,k)=k[x^{\pm 1}]. \]
If $n\geq 3$, then Tate cohomology of $C_{n}$ is the tensor product of an exterior algebra on a $1$-dimensional class $x$ and a Laurent polynomial ring on a $2$-dimensional class $y$,
\[ \hat{H}^*(C_n,k)=\Lambda(x)\otimes k[y^{\pm 1}]. \]
The canonical element
$\gamma_{C_n} \in \hochschild^{3,-1} \hat{H}^*(C_n,k)$
is non-trivial for $n=3$ and trivial for all other values of $n$.
\end{satz}
\begin{proof}
The general method is the following: First of all one constructs a complete projective resolution of $k$ as trivial $kG$-module. Then one finds ('sufficiently many') cocycles of the endomorphism-dga of this projective resolution. Furthermore, one has to construct homotopies representing relations in $\hat{H}^*(G,k)$. From these cocycles and homotopies we define maps $f_1$ and $f_2$ as in §~\ref{kanonischel}. Finally we determine the class of the expression \eqref{mkonstruktion} for all triples $(a,b,c)$.

Let us start with the projective resolution. Instead of $kC_n$ we will use the truncated polynomial algebra $L=k[z]/z^n\cong kC_n$. For all $j$ put $\hat{P}_j=L$, the free $L$-module of rank $1$. Define the differential $\partial:\hat{P}_{j+1}\longrightarrow \hat{P}_j$ to be multiplication by $z$ for $j$ even and multiplication by $-z^{n-1}$ for $j$ odd. 
This is a minimal projective resolution of $k$ as trivial $L$-module. Applying $\Hom_L(\_,k)$ to this sequence we get the complex
\begin{equation} \label{zyklkseq1} 
 \dots \stackrel{0}{\longrightarrow} k \stackrel{0}{\longrightarrow} k \stackrel{0}{\longrightarrow} k \stackrel{0}{\longrightarrow} \dots 
\end{equation}
Hence, $\Tate^j_L(k,k)\cong k$ for all $j\in\mathbb{Z}$. Let $A$ be the endomorphism-dga of $\hat{P}_*$. 

Next we determine the multiplicative structure of $\Tate^j_L(k,k)\cong H^*A$ by writing down cocycles in $A$. 
Let $\Bar x:\hat{P}[1]\longrightarrow \hat{P}$ be the chain map of degree $1$ given by the identity map in even dimensions and multiplication by $z^{n-2}$ in odd dimensions.
This chain map satisfies $d\Bar{x}=0$. The class of $\Bar{x}$ in $H^*A$ will be denoted by $x$. The map $\Bar{x}:\hat{P}_1\longrightarrow \hat{P}_0$ is the identity; hence, under the isomorphism \eqref{tateiso}, $x$ is mapped to a generator of $\Tate_L^1(k,k)\cong k$. Define $\Bar{y}:\hat{P}[2]\longrightarrow \hat{P}$ to be the chain map of degree $2$ given by the identity in every dimension. Since $\hat{P}_*$ is $2$-periodic, this is a cocycle of $A$ representing some class $y\in H^2A$, which is invertible (because $\Bar{y}$ is invertible). Therefore, $\Tate_L^{2i+j}(k,k)$ is generated by $y^ix^j$ (for all $i\in\mathbb{Z},j=0,1$).

The multiplicative structure is determined as soon as we know $x^2$. 
The cocycle $\Bar{x}^2:\hat{P}[2]\longrightarrow\hat{P}$ is given by multiplication by $z^{n-2}$ is every dimension. If $n=2$, this is the same as $\Bar{y}$; therefore $y=x^2$ in this case:
\[ \hat{H}^*(C_2,k)\cong k[x^{\pm 1}] \]
If $n\geq 3$, then $\Bar{x}^2$ is null-homotopic via the homotopy $\Bar{q}$ which is the zero map in even dimensions and multiplication by $z^{n-3}$ in odd dimensions.
Therefore, 
\[ \hat{H}^*(C_n,k)\cong k[x,y^{\pm 1}]/(x^2)\cong \Lambda(x)\otimes k[y^{\pm 1}]. \]
Now we have to choose the maps $f_1$ and $f_2$. In order to define the $k$-linear map $f_1:H^*A\rightarrow A$, it suffices to do this on a $k$-basis. Let us take
\[ f_1(x^\epsilon y^i)=\Bar{x}^\epsilon \Bar{y}^i \]
for all $\epsilon\in\{0,1\}$ and $i\in\mathbb{Z}$. 

If $n=2$ then this is the map $x^i\mapsto \Bar{x}^i$ which is clearly multiplicative. Thus we can choose $f_2=0$, which leads to $\gamma_{C_2}=0$. If $n\geq 3$, we can define $f_2$ as follows:
\begin{align*}
f_2(y^i,y^j)&=f_2(xy^i,y^j)=f_2(y^i,xy^j)=0 \\
f_2(xy^i,xy^j)&=\bar{q}\bar{y}^{i+j} 
\end{align*}
Plugging this into \eqref{mkonstruktion}, one sees that the map $m(a,b,c)$ vanishes unless all $a$, $b$ and $c$ are of odd degree. In this case, $m(xy^i, xy^j, xy^l)$ is represented by
\begin{multline*} f_2(xy^i,xy^j)f_1(xy^l)-f_2(xy^i,x^2y^{j+l})+f_2(x^2y^{i+j},xy^l)-(-1)f_1(xy^i)f_2(xy^j,xy^l) \\
=\bar{q}\bar{x}\bar{y}^{i+j+l}+\bar{x}\bar{q}\bar{y}^{i+j+l}=(\bar{q}\bar{x}+\bar{x}\bar{q})\bar{y}^{i+j+l},
\end{multline*}
which is multiplication by $z^{n-3}$ in all degrees. If $n\geq 4$, this map is null-homotopic (e.g.~via the homotopy which is multiplication by $z^{n-4}$ in odd degrees and zero in even degrees). Hence $m=0$, which implies $\gamma_{C_n}=0$. The fact that $\gamma_{C_3}$ is non-trivial will follow from Theorem~\ref{nichtrealisierbar3}.
\end{proof}

\newcommand{\twoprod}[4]{%
\xymatrix@C=57pt{%
 \hat{P}_2 \ar[d]^{#1}
 & \hat{P}_3 \ar[l]_{\smatrix{I+E \\ J+E}}\ar[d]^{#2}
 & \hat{P}_4 \ar[l]_{\norm}\ar[d]^{#3}
 & \hat{P}_5 \ar[l]_{\smatrix{I+E & J+E}}\ar[d]^{#4}
 & \hat{P}_6 \ar[l]_{\smatrix{J+E&K'+E\\ K+E&I+E}}\ar[d]^{#1} \\
 \hat{P}_0
 & \hat{P}_1 \ar[l]^{\smatrix{I+E & J+E}}
 & \hat{P}_2 \ar[l]^{\smatrix{J+E & K'+E \\ K+E & I+E}}
 & \hat{P}_3 \ar[l]^{\smatrix{I+E \\ J+E}}
 & \hat{P}_4 \ar[l]^{\norm}
}}

\section{The quaternion group $Q_8$}\label{secquagruppe}

Now we apply the theory described in the previous section to a particular group, namely the quaternion group in $8$ elements, given by
\[ G=Q=Q_8=\left< I,J,K \mid I^2=J^2=K^2, IJ=K \right>. \]
The unit we denote by $E$, and the elements usually called $-E$, $-I$, $-J$ and $-K$ are written as $E'$ etc.~to prevent confusion with the addition in the group ring. Thus we have
$Q=\{E,I,J,K,E',I',J',K'\}.$
Let $k$ be a field of characteristic $2$.
Then the Tate cohomology ring $\hat{H}^*(Q)$ is well-known; it is given by
\[ \hat{H}^*(Q)=\Tate_L^*(k,k)\cong k[x,y,s^{\pm 1}]/(x^2+y^2=xy,x^3=y^3=0,x^2y=xy^2) \]
with $|x|=|y|=1, |s|=4$ (see e.g.~\cite{CartEile},~XII § 11, and \cite{AdemMilgram},~IV Lemma 2.10).
The main goal of this section is to prove the following theorem.
\begin{satz}\label{haupttheorem}
The element $\gamma_Q\in \hochschild^{3,-1} \hat{H}^*(Q)$ is non-trivial, and the cokernel of the map
\[ \hat{H}^*(Q)[-1] \oplus \hat{H}^*(Q)[-1] \xrightarrow{\smatrix{y & x+y \\ x & y}} \hat{H}^*(Q)\oplus \hat{H}^*(Q) \]
is a graded $\hat{H}^*(Q)$-module which is not a direct summand of a realisable one.
\end{satz}
The only new thing is the non-triviality of $\gamma_Q$ and the non-realisability of the given module.
We will also explicitly compute a Hochschild cocycle representing $\gamma_Q$. The method is mainly the same as in §~\ref{szyklischegruppen}, but everything will be slightly more complicated.\bigskip

Let $L=kQ$ be the group ring over the field $k$; all modules over $L$ will be right $L$-modules. We often consider free $L$-modules; maps between them will be given in matrix form, with multiplication from the left (in order to make them right $L$-homomorphisms). Chain maps will be denoted with a bar ($\bar{x}$), the corresponding class in $\Tate$ will then be called $x$.

\subsection{The cohomology ring and base homotopies}\label{basehomotopies}

\begin{lemma}
A free and minimal resolution of $k$ (as a trivial $L$-module) is given by
\[
\xymatrix@C=60pt{
k & P_0=L \ar[l]^{\epsilon} & P_1=L^2 \ar[l]^{\partial_1=\smatrix{ I+E & J+E }}  & P_2=L^2\ar[l]^{\partial_2=\smatrix{ J+E&K'+E\\ K+E&I+E}} \\
& & P_3=L \ar[l]^{\partial_3=\smatrix{ I+E \\ J+E }} & P_4=L \ar[l]^{\partial_4=\norm\cdot} & \cdots\ar[l]^{\partial_5}
} 
\]
where $\partial_{n+4}=\partial_n$ and $P_{n+4}=P_n$ for all $n$. Here $\epsilon$ is the augmentation and $\norm$ is the norm element $\norm=\sum_{g\in Q} g\in L$.
\end{lemma}
We omit the straightforward proof.\rem{REFERENZ?!}

By extending this complex $4$-periodically we get the exact sequence $\hat{P}$:
\[ \dots \xleftarrow{\partial_{-2}} \hat{P}_{-2} \xleftarrow{\partial_{-1}} \hat{P}_{-1} \xleftarrow{\partial_0} 
\hat{P}_0 \xleftarrow{\partial_1} \hat{P}_1 \xleftarrow{\partial_2} \dots \]
Since we have taken a minimal resolution, the differential of the complex $\Hom_L(\hat{P}_*,k)$ vanishes; therefore 
$\Tate_L^{4n}(k,k)\cong\Tate_L^{4n+3}(k,k)\cong k$ and
$\Tate_L^{4n+1}(k,k)\cong \Tate_L^{4n+2}(k,k)\cong k^2$ for all integers $n$. 

Let $\bar{s}:\hat{P}[4]\rightarrow \hat{P}$ be the shift map, given by the identity in every degree. This is an invertible cocycle; thus, multiplication by a suitable power of $s$ yields an isomorphism $\Tate_L^{4u+t}(k,k)\cong\Tate_L^{t}(k,k)$ for $t=0,1,2,3$ and $u\in\mathbb{Z}$.

Now we are heading for explicit generators $x,y$ of $\Tate_L^1(k,k)\cong H^1\Hom_L^*(\hat{P},\hat{P})$, which are represented by chain maps $\bar{x},\bar{y}:\hat{P}[1]\rightarrow \hat{P}$. By construction we have $\hat{P}_1=L^2$ and $\hat{P}_0=L$. We extend the two projections $\hat{P}_1\rightarrow \hat{P}_0$ to chain transformations $\hat{P}[1]\rightarrow \hat{P}$ as follows: For $\bar{x}$ we take
\begin{align*}
\xymatrix@C=50pt{
  \hat{P}_1\ar[d]^{\smatrix{E&0}}
   & \hat{P}_2\ar[d]^{\smatrix{ 0&E\\ E&I }}\ar[l]_{\smatrix{J+E&K'+E\\ K+E&I+E}}
   & \hat{P}_3\ar[d]^{\smatrix{ E\\ 0}}\ar[l]_{\smatrix{I+E \\ J+E}}
   & \hat{P}_4\ar[d]^{\begin{smallmatrix} E+E'+J+J' \end{smallmatrix}}\ar[l]_{\norm}
   & \hat{P}_5\ar[d]^{\smatrix{E&0}}\ar[l]_{\smatrix{I+E & J+E}} \\
  \hat{P}_0
   & \hat{P}_1\ar[l]^{\smatrix{I+E & J+E}}
   & \hat{P}_2\ar[l]^{\smatrix{J+E & K'+E \\ K+E & I+E}}
   & \hat{P}_3\ar[l]^{\smatrix{I+E \\ J+E}}
   & \hat{P}_4\ar[l]^{\norm}
}
\end{align*}
extended $4$-periodically in both directions. Similarly, we take as $\bar{y}$:
\[
\xymatrix@C=50pt{
  \hat{P}_1\ar[d]^{\smatrix{0&E}}
   & \hat{P}_2\ar[d]^{\smatrix{ J&E\\ E&0 }}\ar[l]_{\smatrix{J+E&K'+E\\ K+E&I+E}}
   & \hat{P}_3\ar[d]^{\smatrix{ 0\\ E}}\ar[l]_{\smatrix{I+E \\ J+E}}
   & \hat{P}_4\ar[d]^{E+E'+I+I'}\ar[l]_{\norm}
   & \hat{P}_5\ar[d]^{\smatrix{0&E}}\ar[l]_{\smatrix{I+E & J+E}} \\
  \hat{P}_0
   & \hat{P}_1\ar[l]^{\smatrix{I+E & J+E}}
   & \hat{P}_2\ar[l]^{\smatrix{J+E & K'+E \\ K+E & I+E}}
   & \hat{P}_3\ar[l]^{\smatrix{I+E \\ J+E}}
   & \hat{P}_4\ar[l]^{\norm}
}
\]
Since these cocycles are $4$-periodic, they commute with $\bar{s}$. Let us compute their pairwise products: $\bar{x}^2$ is given by
\[ \twoprod{\smatrix{0 & E}}{\smatrix{0 \\ E}}{\smatrix{E+E'+J+J' \\ 0}}{\smatrix{E+E'+J+J' & 0}} \]
$\bar{y}^2$ has the following form:
\[ \twoprod{\smatrix{E & 0}}{\smatrix{E \\ 0}}{\smatrix{0 \\ E+E'+I+I'}}{\smatrix{0 & E+E'+I+I'}} \]
Now $\bar{x}\bar{y}$:
\[ \twoprod{\smatrix{J & E}}{\smatrix{E \\ I}}{\smatrix{E+E'+I+I' \\ 0}}{\smatrix{0 & E+E'+J+J'}} \]
Similarly $\bar{y}\bar{x}$:
\[ \twoprod{\smatrix{E & I}}{\smatrix{J \\ E}}{\smatrix{0 \\ E+E'+J+J'}}{\smatrix{E+E'+I+I' & 0}} \]

In each of these cocycles, the first map $\hat{P}_2\rightarrow \hat{P}_0$ determines the cohomology class by the isomorphism \eqref{tateiso}; in $k^2$, they correspond to $(0\,1), (1\,0), (1\,1)$ and $(1\,1)$, respectively. Hence $\smash{\Tate}^2_L(k,k)$ is generated by $x^2$ and $y^2$, and we have $x^2+y^2=xy=yx$.

But we will need explicit chain homotopies for these equations later on. For $\bar{x}\bar{y}+\bar{y}\bar{x}$ we get
\[ \twoprod{\transpose{\smatrix{E+J \\ E+I}}}{\smatrix{E+J\\ E+I}}{\smatrix{E+E'+I+I' \\ E+E'+J+J'}}{\transpose{\smatrix{E+E'+I+I' \\ E+E'+J+J'}}} \]
As a chain homotopy $\bar{p}$ we can choose:
\[
\xymatrix@C=57pt{%
 \hat{P}_2 \ar[dr]^(0.6){\smatrix{0&E\\ E&0}}
 & \hat{P}_3 \ar[l]_{\smatrix{I+E \\ J+E}}\ar[dr]^(0.55){0}
 & \hat{P}_4 \ar[l]_{\norm}\ar[dr]^(0.55){E+E'}
 & \hat{P}_5 \ar[l]_{\smatrix{I+E & J+E}}\ar[dr]^(0.55){0}
 & \hat{P}_6 \ar[l]_{\smatrix{J+E&K'+E\\ K+E&I+E}} \\
 \hat{P}_0
 & \hat{P}_1 \ar[l]^{\smatrix{I+E & J+E}}
 & \hat{P}_2 \ar[l]^{\smatrix{J+E & K'+E \\ K+E & I+E}}
 & \hat{P}_3 \ar[l]^{\smatrix{I+E \\ J+E}}
 & \hat{P}_4 \ar[l]^{\norm}
}
\]
extended $4$-periodically.
The map $\bar{x}^2+\bar{y}^2+\bar{x}\bar{y}$ is given by:
\[
\twoprod{\transpose{\smatrix{E+J \\ 0}}}{\smatrix{0\\ E+I}}{\smatrix{I+I'+J+J' \\ E+E'+I+I'}}{\transpose{\smatrix{E+E'+J+J' \\ I+I'+J+J'}}}
\]
In this case, there is no $4$-periodic chain homotopy (as we will see in Remark~\ref{keine4periode}). But there is an $8$-periodic one: The maps $(\bar{q}_0,\bar{q}_1,\dots,\bar{q}_7)$ are given (in matrix form) by
\[ \Bigl(
 \smatrix{0 & 0}, 
 \smatrix{0 &  0 \\ E & 0},
 \smatrix{0 \\ 0},
 \smatrix{E+I'+J+K},
 \smatrix{E & E},
 \smatrix{J &  0 \\ E & I},
 \smatrix{E \\ E},
 \smatrix{E+I+J'+K}
\Bigr) \]
We extend this $8$-periodically and obtain a homotopy $\bar{q}$, which in fact satisfies the equality
\[ \bar{s}\bar{q}=(\bar{q}+\bar{x}+\bar{y})\bar{s}. \]

Finally, let us consider $\Tate^3_L(k,k)$. It turns out that $\bar{x}^3$ is null-homotopic, for example via the $8$-periodic homotopy $\bar{r}$ obtained by $8$-periodic extension of $(\bar{r}_0,\bar{r}_1,\dots,\bar{r}_7)$, which are given by
\[
\Bigl(
 \smatrix{0 & 0},
 \smatrix{0 \\ 0},
 \smatrix{E+J' \\ I+J},
 \smatrix{I+J & J+K},
 \smatrix{0 & E},
 \smatrix{0 \\ E},
 \smatrix{E'+J \\ I+J},
 \transpose{\smatrix{E+E'+J'+I\\ J+K}}
\Bigr) \]
The map $\bar{r}$ additionally satisfies
\[ \bar{s}\bar{r}=(\bar{r}+\bar{x}^2)\bar{s}. \]
Since $x^3+y^3=(x+y)(x^2+xy+y^2)$, $\bar{y}^3$ is null-homotopic as well. Furthermore, $xy^2=(xy)y=(x^2+y^2)y=x^2y+y^3=x^2y$. By inspection, one sees that in degree $0$, the map $\bar{x}^2\bar{y}:\hat{P}[3]\rightarrow \hat{P}$ is given by
$\hat{P}_3=L\xrightarrow{\Id} L=\hat{P}_0$; therefore $x^2y$ generates $\Tate^3_L(k,k)$. Taking all these results together, we have verified that the ring structure of $\hat{H}^*(Q)$ is the one given in Theorem~\ref{haupttheorem}. Let us remark here that all monomials in $x$ and $y$ of degree bigger than $3$ vanish in this ring.

\subsection{The class of a map}\label{sklasse}
Let us recall the construction of a representative of $\gamma_{Q_8}$. First of all, we have to choose a cycle selection-homomorphism
$ f_1:\Tate_L^*(k,k)\rightarrow \Hom_L^*(\hat{P},\hat{P}) $
such that any class $a$ is mapped to a representative $f_1(a)$. Then we can find a $k$-linear map
$ f_2:\Tate_L^*(k,k)\otimes\Tate_L^*(k,k)\rightarrow \Hom_L^*(\hat{P},\hat{P}) $
of degree $-1$ satisfying
$ df_2(a,b)=f_1(a)f_1(b)-f_1(ab) $
for all $a,b$. Finally, we are interested in terms of the form
\begin{equation}\label{kozykelterm} f_2(a,b)f_1(c)+f_2(a,bc)+f_2(ab,c)+f_1(a)f_2(b,c); \end{equation}
this is a cocycle in $\Hom_L^*(\hat{P},\hat{P})$. In order to determine the class of this cocycle, it is enough to know the degree $0$ map of it (cf.~\eqref{tateiso}). This observation leads to the following definition.
\begin{definition}
For every $f\in \Hom^n_L(\hat{P},\hat{P})$, i.e.,~a family of maps $f_j:\hat{P}_{j+n}\rightarrow \hat{P}_j$ ($j\in\mathbb{Z}$), not necessarily commuting with the differential, we denote by $\class(f)$ the class of the map 
$\epsilon \circ f_0: \hat{P}_n\rightarrow k$ in $H^n \Hom_L(\hat{P}_*,k)=\Tate_L^n(k,k)$. 
\end{definition}
Note that the complex $\Hom_L(\hat{P}_*,k)$ has trivial differential; thus, every element in $\Hom_L(\hat{P}_*,k)$ and in particular $\epsilon\circ f_0$ is a cocycle.
The definition above gives a map
\[
\boxed{
\begin{aligned}
 \class:\Hom^n_L(\hat{P},\hat{P}) &\longrightarrow  \Tate^n_L(k,k) \\
  f & \;\mapsto\;  [\epsilon\circ f_0]
\end{aligned}
}
\]
\begin{lemma} \label{classlemma}
The map $\class$ has the following properties:
\begin{enumerate}
\item[(i)] If $f\in\Hom^n_L(\hat{P},\hat{P})$ is a cocycle, then $\class(f)$ is the cohomology class of $f$; in particular $\class\circ f_1=\Id$.
\item[(ii)] The map $\class$ is $k$-linear.
\item[(iii)] If $\class(f_1)=\class(f_2)$ for some $f_1,f_2\in\Hom^n_L(\hat{P},\hat{P})$, then $\class(f_1g)=\class(f_2g)$ for all $g\in\Hom^m_L(\hat{P},\hat{P})$.
\item[(iv)] If $a\in\Hom^m_L(\hat{P},\hat{P})$ is a cocycle and $f\in\Hom^n_L(\hat{P},\hat{P})$ is arbitrary, then $\class(fa)=\class(f)\class(a)$.
\end{enumerate}
\end{lemma}
\begin{proof}
(i) follows from \eqref{tateiso}, (ii) holds by definition. \\
Ad (iii): If $\class(f_i)=0$, then $\epsilon\circ f_i=0$. This implies $\epsilon\circ f_i\circ g=0$, hence $\class(f_ig)=0$. For general $f_1,f_2$ note $\class(f_1-f_2)=0$; by what we just proved $\class((f_1-f_2)g)=0$ and therefore $\class(f_1g)=\class(f_2g)$.\\
Ad (iv): Choose a cocycle $h\in\Hom^n_L(P,P)$ satisfying $\class(h)=\class(f)$. Then by (iii)
\[ \class(fa)=\class(ha)=\class(h)\class(a)=\class(f)\class(a). \qedhere \]
\end{proof}

The following corollary will simplify computations later on.
\begin{lemma} \label{vereinfachlemma}
The map $f_2$ can be chosen in such a way that $\class\circ f_2=0$.
\end{lemma}
\begin{proof}
Choose any $\tilde{f_2}$ (satisfying $d\tilde{f}_2(a,b)=f_1(a)f_1(b)-f_1(ab)$). Put $f_2=\tilde{f}_2-f_1\circ \class \circ \tilde{f}_2$. Since $df_1=0$, we get
\[ df_2(a,b)=d\tilde{f}_2(a,b)=f_1(a)f_1(b)-f_1(ab), \]
and from $\class\circ f_1=\Id$ follows that
\begin{equation*} \class\circ f_2=\class\circ \tilde{f}_2-\class\circ f_1\circ \class\circ \tilde{f}_2=0. \qedhere \end{equation*}
\end{proof}

Consider (\ref{kozykelterm}) with this simplified version of $f_2$. By applying $\class$, we get the term
\[ \class(f_2(a,b)f_1(c))+\class(f_2(a,bc))+\class(f_2(ab,c))+\class(f_1(a)f_2(b,c)) \]
This is the cohomology class of \eqref{kozykelterm}. Note that the individual terms $f_2(a,b)f_1(c)$, $f_2(a,bc)\dots$ will not be cocycles in general, but the map $\class$ assigns cohomology classes to them in such a way that the sum will be the class we are looking for.

By our choice of $f_2$ (such that $\class\circ f_2=0$), the first three terms in the sum vanish (note that $\class(f_2(a,b)f_1(c))=\class(f_2(a,b))c$ by Proposition~\ref{classlemma}.(iv)). Thus we are interested in terms of the form
$\class(f_1(a)f_2(b,c))$, where $a,b,c$ run through all elements of a $k$-basis of $\smash{\Tate}_L^*(k,k)$.\bigskip

\begin{lemma} \label{pqreigenschaften}
For all monomials $\alpha,\beta$ in $\bar{x},\bar{y}$ we have:
\begin{align*}
 \class(\bar{p}\alpha)&=0,  
 &\class(\bar{x}\bar{p}\alpha)&=x^2\class(\alpha), \\
 \class(\bar{q}\alpha)&=0,
 &\class(\bar{y}\bar{p}\alpha)&=y^2\class(\alpha), \\
 \class(\bar{x}\bar{q}\alpha)&=0,
 &\class(\bar{y}\bar{q}\alpha)&=y^2\class(\alpha), \\
 \class(\beta\bar{p}\alpha)&=0 &&\text{if the degree of $\beta$ is at least $2$} \\
 \class(\beta\bar{q}\alpha)&=0 &&\text{if the degree of $\beta$ is at least $2$} \\
 \class(\beta\bar{r}\alpha)&=0 &&\text{for all $\beta$}
\end{align*}
Here, $\bar{p},\bar{q}$ and $\bar{r}$ are defined as in §~\ref{basehomotopies}.
\end{lemma}
\begin{proof}
By Proposition~\ref{classlemma}.(iii) we can assume that the degree of $\beta$ is at most $3$. Furthermore, we can assume $\alpha=1$ by Proposition~\ref{classlemma}.(iv). In order to determine $\class(\bar{a}\bar{p})$ for any given cocycle $\bar{a}$ of degree $n$, we consider the composition
\begin{equation*}
  \hat{P}_{n+1} \xrightarrow{\bar{p}_n} \hat{P}_n \xrightarrow{\bar{a}_0} \hat{P}_0 \xrightarrow{\hspace{1pt}\epsilon} k 
\end{equation*}
as an element of $H^{n+1}\Hom_L(\hat{P}_*,k)$. 
If $n=0,2,3$, then $\im(\bar{p}_n)\subset \ker(\epsilon)\cdot \hat{P}_n$. Therefore, $\im(\bar{a}_0\circ\bar{p}_n)\subset \ker(\epsilon)\cdot \hat{P}_0=\ker(\epsilon)$, hence  $\epsilon\circ \bar{a}_0\circ \bar{p}_n=0$. 
For $\class(\bar{x}\bar{p})$ consider $\bar{x}\bar{p}$ in degree $0$, i.e.~
\[
\begin{array}{ccccc}
\hat{P}_2 &\stackrel{\bar{p}_1}{\longrightarrow}&\hat{P}_1&\stackrel{\bar{x}_0}{\longrightarrow}& \hat{P}_0 \\
&\smatrix{0 & E \\ E & 0}&&\smatrix{E & 0}
\end{array}
\]
that is $\smatrix{0 & E}:\hat{P}_2\longrightarrow \hat{P}_0$, which corresponds to $x^2$. The remaining cases can be shown analogously.
\end{proof}

\begin{bemerkung}\label{keine4periode} Using $\class$, we can prove that there is no $4$-periodic null-homotopy for $\bar{x}^2+\bar{x}\bar{y}+\bar{y}^2$ as follows: Suppose there is a $4$-periodic null-homotopy; call it $\hat{q}$. Since $d(\hat{q}-\bar{q})=0$, $\bar{v}=\hat{q}-\bar{q}$ is a cocycle, representing some class $v$. By construction,  $\bar{s}\bar{q}=(\bar{q}+\bar{x}+\bar{y})\bar{s}$. Since $\hat{q}$ is $4$-periodic, we have
$\class(\bar{s}\bar{v})=\class(\bar{v}\bar{s})-\class((\bar{x}+\bar{y})\bar{s})=vs-(x+y)s$
by Proposition~\ref{classlemma}. On the other hand, $\class(\bar{s}\bar{v})=sv$, hence $(x+y)s=0$, a contradiction. In a similar way one shows that there is no $4$-periodic null-homotopy for\rem{?} $\bar{x}^3$.
\end{bemerkung}

\subsection{The maps $f_1$ and $f_2$}
A $k$-basis of $\Tate$ is given by $\mathfrak{C}=\{s^i, xs^i, ys^i, x^2s^i, y^2s^i, x^2ys^i\mid i\in\mathbb{Z}\}$. Define the $k$-linear map $f_1$ on the basis $\mathfrak{C}$ by
\begin{eqnarray*}
f_1\,:\, \Tate_L^*(k,k) &\rightarrow&\Hom_L^*(\hat{P},\hat{P}) \\
x^\varepsilon y^\delta s^i &\mapsto& \bar{x}^\varepsilon \bar{y}^\delta \bar{s}^i
\end{eqnarray*}
for all  $i,\varepsilon,\delta\in\mathbb{Z}$ for which the expression on the left hand side lies in $\mathfrak{C}$. Put $\mathcal{B}=\{1,x,y,x^2,y^2,x^2y\}$. For all $b,c\in \mathcal{B}$ and $i,j\in\mathbb{Z}$ we have $f_1(bs^ics^j)=f_1(bc)\bar{s}^{i+j}$ and $f_1(bs^i)f_1(cs^i)=f_1(b)f_1(c)\bar{s}^{i+j}$, since $\bar{s}$ commutes with both $\bar{x}$ and $\bar{y}$. This implies that we can define $f_2$ on $\mathcal{B}\times\mathcal{B}$ and then extend it to $\mathfrak{C}\times\mathfrak{C}$ via
$f_2(bs^i,cs^j)=\bar{s}^{i+j}f_2(b,c)$. Now define $f_2$ on $\mathcal{B}\times\mathcal{B}$ as follows:
\begin{flushleft}
\begin{tabular}{cc|cccc|} 
 \multicolumn{2}{c}{\multirow{2}*{$f_2(b,c)$}}    &     \multicolumn{4}{c}{$c$}  \\
\multicolumn{1}{c}{}     &      &$1$&$x$         &$y$         &$x^2$         \\  \cline{2-6}
\multirow{6}*{$b$}    &$1$   &$0$&$0$         &$0$         &$0$             \\
    &$x$   &$0$&$0$         &$\Bar{q}$&$\Bar{r}$       \\
    &$y$   &$0$&$\Bar{p}+\Bar{q}$&$0$         &$\Bar{p}\Bar{x}+\Bar{x}\Bar{p}+\Bar{x}^2$ \\
    &$x^2$ &$0$&$\Bar{r}$       &$0$         &$\Bar{r}\Bar{x}$     \\   
    &$y^2$ &$0$&$\Bar{r}+\Bar{q}\Bar{x}+\Bar{x}\Bar{p}+\Bar{x}^2$ &$\Bar{x}\Bar{q}+\Bar{q}\Bar{y}+\Bar{r}$       &$\Bar{q}\Bar{x}^2+\Bar{r}\Bar{x}+\Bar{p}\Bar{x}^2+\Bar{y}\Bar{r}$   \\
    &$x^2y$&$0$&$\Bar{x}^2\Bar{p}+\Bar{r}\Bar{y}$     &$\Bar{r}\Bar{x}+\Bar{r}\Bar{y}+\Bar{x}^2\Bar{q}$    &$\Bar{x}^2\Bar{p}\Bar{x}+\Bar{r}\Bar{y}\Bar{x}$ \\  \cline{2-6}
\end{tabular}
\end{flushleft}
\begin{flushright}
\begin{tabular}{cc|cc|} 
 \multicolumn{2}{c}{}    &     \multicolumn{2}{c}{}  \\
 \multicolumn{1}{c}{} &    &$y^2$      &$x^2y$    \\  \cline{2-4}
\multirow{6}*{$b$}    &$1$    &$0$        &$0$       \\
    &$x$   &$\Bar{x}\Bar{q}+\Bar{r}$&$\Bar{r}\Bar{y}$    \\
    &$y$   &$\Bar{x}\Bar{q}+\Bar{q}\Bar{y}+\Bar{r}$      &$\Bar{p}\Bar{x}\Bar{y}+\Bar{x}\Bar{p}\Bar{y}+\Bar{x}^2\Bar{q}+\Bar{r}\Bar{y}+\Bar{r}\Bar{x}+\Bar{x}^2\Bar{y}$   \\
    &$x^2$ &$\Bar{r}\Bar{x}+\Bar{r}\Bar{y}+\Bar{x}^2\Bar{q}$   &$\Bar{r}\Bar{x}\Bar{y}$    \\
    &$y^2$ &$\Bar{x}\Bar{q}\Bar{y}+\Bar{q}\Bar{y}^2+\Bar{r}\Bar{y}$      &$\Bar{q}\Bar{x}^2\Bar{y}+\Bar{r}\Bar{x}\Bar{y}
    +\Bar{p}\Bar{x}^2\Bar{y}+\Bar{y}\Bar{r}\Bar{y}$ \\
    &$x^2y$&$\Bar{x}^2\Bar{q}\Bar{y}+\Bar{r}\Bar{x}\Bar{y}+\Bar{r}\Bar{y}^2$   &$\Bar{x}^2\Bar{p}\Bar{x}\Bar{y}+\Bar{r}\Bar{y}\Bar{x}\Bar{y}$\\  \cline{2-4}
\end{tabular}
\end{flushright}
Direct verification shows that $df_2(b,c)=f_1(bc)-f_1(b)f_1(c)$. In fact, this $f_2$ is already simplified in the sense of Proposition~\ref{vereinfachlemma}, which is why some apparently unnecessary terms occur (e.g. the $\bar{x}^2$ in $f_2(y,x^2)$). Indeed, $\class\circ f_2=0$, as one can check using Proposition~\ref{pqreigenschaften}. 

\subsection{Computation of $m$}
We want to investigate\rem{?} the term
\[ m(a,b,c)=\class(f_1(a)f_2(b,c)) \]
for all $a,b,c\in\mathfrak{C}$. Since $f_2(b,c)$ is $8$-periodic, we have
$m(as^{2h},bs^i,cs^j)=m(a,b,c)s^{2h+i+j}$ for all integers $h,i,j$ and $a,b,c\in\mathfrak{C}$. Therefore it is enough to consider all triples $(a,b,c)\in\bigl(\mathcal{B}\cup\mathcal{B} s\bigr)\times\mathcal{B}\times\mathcal{B}$.\bigskip

Consider the case $a\in\mathcal{B}$. If $a=1$, then
$\class(f_1(a)f_2(b,c))=\class(f_2(b,c))=0$. If $|a|\geq 2$, then $f_1(a)f_2(b,c)$ is a sum of terms $\beta\bar{p}\alpha$, $\beta\bar{q}\alpha$ and $\beta\bar{r}\alpha$ where $\alpha$ and $\beta$ are monomials in $\bar{x}$ and $\bar{y}$, and the degree of $\beta$ is at least $2$. Hence $\class(f_1(a)f_2(b,c))=0$ by Proposition~\ref{pqreigenschaften}.

We are left with the cases $a=x$ and $a=y$. Consider $a=x$. By Proposition~\ref{pqreigenschaften} we get $\class(\bar{x}f_2(b,c))$ from $f_2(b,c)$ by the following rule: Put an $\bar{x}$ in front of all monomials in $\bar{x}$ and $\bar{y}$. Then remove all summands containing $\bar{p}$, $\bar{q}$ or $\bar{r}$, except those beginning with $\bar{p}$, where we replace the $\bar{p}$ by $x^2$. Finally, replace all $\bar{x}$ and $\bar{y}$ by $x$ and $y$, respectively. Using this procedure, we get the following table for $\class(\bar{x}f_2(b,c))$:
\begin{center}
\begin{tabular}{cc|cccccc|} 
\multicolumn{2}{c}{\multirow{2}*{$\class(\bar{x}f_2(b,c))$}}  & \multicolumn{6}{c}{$c$}     \\
\multicolumn{1}{c}{}    &      &$1$&$x$         &$y$         &$x^2$       &$y^2$      &$x^2y$    \\  \cline{2-8}
\multirow{6}*{$b$}    &$1$   &$0$&$0$         &$0$         &$0$         &$0$        &$0$   \\
    &$x$   &$0$&$0$         &$0$         &$0$         &$0$        &$0$       \\
    &$y$   &$0$&$x^2$       &$0$         &$x^3+x^3$   &$0$        &$x^3y+x^3y$ \\
    &$x^2$ &$0$&$0$         &$0$         &$0$         &$0$        &$0$       \\
    &$y^2$ &$0$&$x^3$       &$0$         &$x^4$       &$0$        &$x^4y$    \\
    &$x^2y$&$0$&$0$         &$0$         &$0$         &$0$        &$0$       \\  \cline{2-8}
\end{tabular}
\end{center}
Most of these expressions vanish, the only remaining term is
$m(x,y,x)=x^2.$
For the case $a=y$ we use a similar method resulting from Proposition~\ref{pqreigenschaften}, and we end up with the following:
\begin{center}
\begin{tabular}{cc|cccccc|} 
\multicolumn{2}{c}{\multirow{2}*{$\class(\bar{y}f_2(b,c))$}}  & \multicolumn{6}{c}{$c$}     \\
\multicolumn{1}{c}{} &      &$1$&$x$         &$y$         &$x^2$       &$y^2$      &$x^2y$    \\  \cline{2-8}
\multirow{6}*{$b$} &$1$   &$0$&$0$         &$0$         &$0$         &$0$        &$0$       \\
    &$x$   &$0$&$0$         &$y^2$       &$0$         &$0$        &$0$       \\
    &$y$   &$0$&$y^2+y^2$   &$0$         &$y^2x+yx^2$ &$y^3$      &$y^2xy+yx^2y$ \\
    &$x^2$ &$0$&$0$         &$0$         &$0$         &$0$        &$0$       \\
    &$y^2$ &$0$&$y^2x+yx^2$ &$y^3$     &$y^2x^2+y^2x^2$ &$y^4$ &$y^2x^2y+y^2x^2y$       \\
    &$x^2y$&$0$&$0$         &$0$         &$0$         &$0$        &$0$       \\  \cline{2-8}
\end{tabular}
\end{center}
Again, only one expression is non-zero, namely
$m(y,x,y)=y^2.$

The case $a\in\mathcal{B}s$ is slightly more difficult. Consider the map
\[ h(b,c)=\Bar{s}f_2(b,c)\Bar{s}^{-1}-f_2(b,c), \]
measuring how far away $f_2$ is from $4$-periodicity.\rem{?}
From the equations
\begin{eqnarray*}
\Bar{s}\Bar{p}\Bar{s}^{-1}&=&\Bar{p} \\
\Bar{s}\Bar{q}\Bar{s}^{-1}&=&\Bar{q}+\Bar{x}+\Bar{y} \\
\Bar{s}\Bar{r}\Bar{s}^{-1}&=&\Bar{r}+\Bar{x}^2  
\end{eqnarray*}
we get the following table for $h$:
\begin{flushleft}
\begin{tabular}{cc|cccc|} 
\multicolumn{2}{c}{\multirow{2}*{$\bar{h}(b,c)$}}    & \multicolumn{4}{c}{$c$}   \\
\multicolumn{1}{c}{} &  &$1$&$x$         &$y$         &$x^2$         \\  \cline{2-6}
\multirow{6}*{$b$}    &$1$   &$0$&$0$         &$0$         &$0$             \\
    &$x$   &$0$&$0$         &$\Bar{x}+\Bar{y}$&$\Bar{x}^2$       \\
    &$y$   &$0$&$\Bar{x}+\Bar{y}$&$0$         &$0$ \\
    &$x^2$ &$0$&$\Bar{x}^2$       &$0$         &$\Bar{x}^2\Bar{x}$     \\   
    &$y^2$ &$0$&$\Bar{x}^2+(\Bar{x}+\Bar{y})\Bar{x}$ &$\Bar{x}(\Bar{x}+\Bar{y})+(\Bar{x}+\Bar{y})\Bar{y}+\Bar{x}^2$       &$(\Bar{x}+\Bar{y})\Bar{x}^2+\Bar{x}^2\Bar{x}+\Bar{y}\Bar{x}^2$   \\
    &$x^2y$&$0$&$\Bar{x}^2\Bar{y}$     &$\Bar{x}^2\Bar{x}+\Bar{x}^2\Bar{y}+\Bar{x}^2(\Bar{x}+\Bar{y})$    &$\ast$ \\  \cline{2-6}
\end{tabular}
\end{flushleft}
\begin{flushright}
\begin{tabular}{cc|cc|}
\multicolumn{2}{c}{\multirow{2}*{}}    & \multicolumn{2}{c}{}   \\
\multicolumn{1}{c}{}  &      &$y^2$      &$x^2y$    \\  \cline{2-4}
\multirow{6}*{$b$}    &$1$    &$0$        &$0$       \\
    &$x$   &$\Bar{x}(\Bar{x}+\Bar{y})+\Bar{x}^2$&$\Bar{x}^2\Bar{y}$    \\
    &$y$   &$\Bar{x}(\Bar{x}+\Bar{y})+(\Bar{x}+\Bar{y})\Bar{y}+\Bar{x}^2$      &$\Bar{x}^2(\Bar{x}+\Bar{y})+\Bar{x}^2\Bar{y}+\Bar{x}^2\Bar{x}$   \\
    &$x^2$ &$\Bar{x}^2\Bar{x}+\Bar{x}^2\Bar{y}+\Bar{x}^2(\Bar{x}+\Bar{y})$   &$\ast$    \\
    &$y^2$ &$\Bar{x}(\Bar{x}+\Bar{y})\Bar{y}+(\Bar{x}+\Bar{y})\Bar{y}^2+\Bar{x}^2\Bar{y}$      &$\ast$ \\
    &$x^2y$&$\ast$   &$\ast$\\  \cline{2-4}
\end{tabular}
\end{flushright}
where $\ast$ denotes certain homogeneous polynomials in $\bar{x}$ and $\bar{y}$ of degree at least $4$. Applying $\class$ to this table and using relations in $\Tate^*_L(k,k)$, we get
\begin{center}
\begin{tabular}{cc|cccccc|} 
\multicolumn{2}{c}{\multirow{2}*{$\class(h(b,c))$}}  &  \multicolumn{6}{c}{$c$} \\
\multicolumn{1}{c}{}   &      &$1$&$x$         &$y$         &$x^2$  &$y^2$      &$x^2y$        \\ \cline{2-8}
\multirow{6}*{$b$}    &$1$   &$0$&$0$         &$0$         &$0$  & $0$       &$0$          \\
    &$x$   &$0$&$0$         &$x+y$ & $x^2$ & $y^2+x^2$&$x^2y$ \\
    &$y$   &$0$&$x+y$&$0$         &$0$  & $y^2$       &$0$   \\
    &$x^2$ &$0$&$x^2$       &$0$         &$0$  &$0$   &$0$      \\
    &$y^2$ &$0$&$y^2+x^2$ &$y^2$       &$0$  &$0$   &$0$      \\
    &$x^2y$&$0$&$x^2y$     &$0$    &$0$   &$0$   &$0$     \\  \cline{2-8}
\end{tabular}
\end{center}

By definition of $h$ we have $h(b,c)\bar{s}=\bar{s}f_2(b,c)-f_2(b,c)\bar{s}$, hence
\[ \class(h(b,c))s=\class(\Bar{s}f_2(b,c))-\underbrace{\class(f_2(b,c))}_{0} s =m(s,b,c). \]
Therefore, this table shows the values $m(s,b,c)$ with $b,c\in\mathcal{B}$. On the other hand, we know that $m$ is a Hochschild-cocycle, in particular for all $a,b,c\in\mathcal{B}$
\[ a\, m(s,b,c)+m(as,b,c)+m(a,sb,c)+m(a,s,bc)+m(a,s,b)c =0. \]
Using $m(a,s,b)c=m(a,1,b)sc=0$, $m(a,s,bc)=m(a,1,bc)s=0$ and $m(a,sb,c)=m(a,b,c)s$, we get
\begin{equation}\label{mgleichung} m(as,b,c)=a\, m(s,b,c)+m(a,b,c)s \end{equation}
We know the right hand side for all $a,b,c\in\mathcal{B}$. Gathering all results, we get the following theorem.
\begin{satz}\label{quatemaintheorem} The canonical element $\gamma_Q$ is represented by the Hochschild cocycle $m$ which satisfies $m(as^{2h},bs^i,cs^j)=m(a,b,c)s^{2h+i+j}$,
\begin{center}
\begin{tabular}{l|l|l} 
$m(x,y,x)=x^2$ &          $m(s,x,x^2y)=x^2ys$       &  $m(sy,x,x^2)=x^2ys$ \\
$m(y,x,y)=y^2$ &          $m(s,x^2y,x)=x^2ys$       &  $m(sy,x^2,x)=x^2ys$\\
$m(s,x,y)=(x+y)s$ &       $m(sx,y,x)=(x^2+y^2)s$    &  $m(sx,x,y^2)=x^2ys$ \\
$m(s,y,x)=(x+y)s$ &       $m(sy,x,y)=(x^2+y^2)s$    &  $m(sy,x,y^2)=x^2ys$\\
$m(s,x,x^2)=x^2s$ &       $m(sx,x,y)=y^2s$          &  $m(sx,y^2,x)=x^2ys$  \\
$m(s,x^2,x)=x^2s$ &       $m(sy,y,x)=x^2s$          &  $m(sy,y^2,x)=x^2ys$ \\
$m(s,x,y^2)=(x^2+y^2)s$ & $m(sx^2,x,y)=x^2ys$       &  $m(sx,y,y^2)=x^2ys$ \\
$m(s,y^2,x)=(x^2+y^2)s$ & $m(sy^2,x,y)=x^2ys$       &  $m(sx,y^2,y)=x^2ys$ \\
$m(s,y,y^2)=y^2s$ &       $m(sx^2,y,x)=x^2ys$       &   \\
$m(s,y^2,y)=y^2s$ &       $m(sy^2,y,x)=x^2ys$       &   \\
\end{tabular}
\end{center}
and vanishes on all other triples of elements of $\mathfrak{C}$.
\end{satz}

\begin{satz}
The element $\gamma_Q$ is non-trivial.
\end{satz}
\beweis
Assume $m=dg$ for some Hochschild $(2,-1)$-cochain $g$. Then,
\[ m(a,b,c)=(dg)(a,b,c)=a\, g(b,c)+g(ab,c)+g(a,bc)+g(a,b)c \]
for all $a,b,c$. In particular,
\begin{align*}
y^2=m(y,x,y)&=yg(x,y)+g(yx,y)+g(y,xy)+g(y,x)y  \\
0=m(x,y,y)&=xg(y,y)+g(xy,y)+g(x,y^2)+g(x,y)y \\
0=m(x,x,x)&=xg(x,x)+g(x^2,x)+g(x,x^2)+g(x,x)x \\
x^2=m(x,y,x)&=xg(y,x)+g(xy,x)+g(x,yx)+g(x,y)x \\
0=m(y,y,x)&=yg(y,x)+g(y^2,x)+g(y,yx)+g(y,y)x
\end{align*}
Adding up these equations we get (using $x^2+y^2=xy$)
\[ x^2+y^2=x\cdot (g(x,y)+g(y,x)). \]
This implies $g(x,y)+g(y,x)=y$. On the other hand, interchanging the roles of $x$ and $y$ we get $g(x,y)+g(y,x)=x$, a contradiction.\qed

\subsection{Matric Massey products}
In order to construct a module which is not a direct summand of a realisable one, we will use the calculus of matric Massey products (introduced by May, \cite{May}). Let us recall the basic definitions and properties. For simplicity, we shall stick to the case of characteristic $2$. Let $A$ be a differential graded algebra over the field $k$. We denote by $\mathcal{F}(A)$ the set of all (right) $A$-modules of the form
$\bigoplus_{\mu=1}^m A[m_\mu]$ (for some natural number $m$ and some integers $m_1,m_2,\dots,m_m$). 

\begin{definition}
For any two dg-$A$-modules $P,Q$ let $\Mat(P,Q)$ be the set of all 
$A$-module-homomorphisms from $P$ to $Q$. If
\[ P=\bigoplus_{\mu=1}^m A[-m_\mu] \in \mathcal{F}(A), \]
we identify $\Mat(P,Q)$ with those $1\times m$-matrices $X$ having entries in $Q$ and satisfying $|X_\mu|=m_\mu$. If further
\[ Q=\bigoplus_{\nu=1}^n A[-n_\nu] \in \mathcal{F}(A), \]
we identify $\Mat(P,Q)$ with the $n\times m$-matrices $X$ having entries in $Q$ and satisfying $|X_{\nu,\mu}|=m_\mu-n_\nu$. For all such matrices we define
$(dX)_{\nu,\mu} = dX_{\nu,\mu}.$
\end{definition}

Now we are ready to define matric Massey products. 
Let $P,Q,R\in \mathcal{F}(H^*A)$ and let $M$ be an arbitrary dg-$A$-module. Suppose we are given maps
\[ R\stackrel{Y}{\longrightarrow} Q\stackrel{X}{\longrightarrow} P\stackrel{W}{\longrightarrow} H^*M \]
(represented by matrices $Y$ and $X$ with entries in $H^*A$ and a vector $W$ with entries in $H^*M$) such that $WX=0$ and $XY=0$. By choosing a representative in $M$ for every entry of $W$ we obtain a matrix $\bar{W}$; similarly we can choose matrices $\bar{X}$ and $\bar{Y}$ with entries in $A$. Then there are matrices $\bar{T}:\bar{Q}\rightarrow M[1]$ and $\bar{U}:\bar{R}\rightarrow \bar{P}[1]$ with entries in $M$ and $A$ respectively, satisfying $\bar{W}\bar{X}=d\bar{T}$ and $\bar{X}\bar{Y}=d\bar{U}$. Then all entries of the matrix $\bar{B}=\bar{T}\bar{Y}-\bar{W}[1]\bar{U}$ (as a map $\bar{R}\rightarrow M[1]$) are cocycles; it therefore represents a matrix $B$ with entries in $H^*M$. The set of matrices obtained this way is a coset of $W[1]\Mat(R,P[1])+\Mat(Q,H^*M[1])Y$ in $\Mat(R,H^*M[1])$, which will be denoted by $\left<W,X,Y\right>$.

\begin{lemma}[Juggling formula] Suppose we are additionally given a module $S\in\mathcal{F}(H^*A)$ and a morphism $Z:S\longrightarrow R$ satisfying $YZ=0$. Then
\[ W\left<X,Y,Z\right>=\left<W,X,Y\right>Z \]
as cosets of $W\cdot\Mat(R,P[1])\cdot Z$ in $\Mat(S,H^*M[1])$.
\end{lemma}
This is a special case of Corollary 3.2.(iii) of \cite{May}. 

\begin{lemma} \label{cokernlemma}
Let $P,Q,R,S$ be as above. Suppose we are given maps
\[ S\stackrel{Z}{\longrightarrow} R\stackrel{Y}{\longrightarrow}
   Q\stackrel{X}{\longrightarrow} P \]
satisfying $XY=0$ and $YZ=0$. Furthermore, assume $0\not\in
\left<X,Y,Z\right>$. Then $C=\coker(X)$ is not a direct summand of a realisable module.
\end{lemma}
\beweis
Assume there were some dg-$A$-module $M$ and maps
\[ C\stackrel{i}{\longrightarrow} H^*M \stackrel{r}{\longrightarrow}
C, \] 
such that $r\circ i=\id_C$. Define $W=i\circ\pi$, where $\pi:P\rightarrow C$ is the projection.
Then $WX=i\circ \pi\circ X=0$, and by the juggling formula we have $\left<W,X,Y\right>Z=W\left<X,Y,Z\right>$ as cosets of $W\cdot \Mat(R,P[1])\cdot Z$ in
$\Mat(S,H^*M[1])$. Let $D\in \left<W,X,Y\right>$. There exists some $E\in\left<X,Y,Z\right>$ satisfying $DZ=WE$. 
\[
\xymatrix@=30pt{
Q[1]\ar[r]^X & P[1]\ar[r]^\pi\ar[dr]^W & C[1] \\
S\ar@{-->}[u]^G\ar@{-->}[ur]_E\ar[r]_Z & R\ar@{-->}[u]_F\ar[r]_-D & H^*M[1]\ar[u]_r } \]
Precomposition with $r$ yields
\[ r\circ D\circ Z=(r\circ W)\circ E=(r\circ i\circ \pi)\circ E
 =\pi\circ E. \]
Since $\pi$ is surjective and $R$ is projective, there is some $F:R\rightarrow P[1]$ such that
$\pi\circ F=r\circ D$. Then the image of $F\circ Z-E$ lies in the kernel of $\pi$, which is the image of $X$. Since $S$ is projective, there is some $G:S\rightarrow Q[1]$ satisfying $F\circ Z-E=X\circ G$. But then
\[ E=F\circ Z-X\circ G\in \Mat(R,P[1])\cdot Z+X\cdot
\Mat(S,Q[1]), \]
But now ~$E\in\left<X,Y,Z\right>$ implies $0\in\left<X,Y,Z\right>$, a contradiction.\qed

Now let $A$ be the endomorphism algebra of our projective resolution $\hat{P}_*$; let us write $\Lambda=H^*A$.
\begin{lemma}
We have
\[ 0\not\in\left<\begin{pmatrix} y & x+y \\ x & y \end{pmatrix},\begin{pmatrix} y & x+y \\ x & y \end{pmatrix},\begin{pmatrix} y & x+y \\ x & y \end{pmatrix} \right>. \]
Therefore, by Proposition~\ref{cokernlemma}, the cokernel of the map
\[ \Lambda[-1] \oplus \Lambda[-1] \xrightarrow{\smatrix{y & x+y \\ x & y}} \Lambda\oplus\Lambda \]
is a graded $\hat{H}^*(G,k)$-module which is not a direct summand of a realisable one.
\end{lemma}
\beweis
We choose $\bar{x}$ and $\bar{y}$ as representatives for $x$ and $y$. Then
\[ \begin{pmatrix} \bar{y} & \bar{x}+\bar{y} \\ \bar{x} & \bar{y} \end{pmatrix} ^2 
 = \begin{pmatrix} \bar{y}^2+\bar{x}^2+\bar{y}\bar{x} & \bar{y}\bar{x}+\bar{x}\bar{y} \\
    \bar{y}\bar{x}+\bar{x}\bar{y} & \bar{y}^2+\bar{x}^2+\bar{x}\bar{y} \end{pmatrix} 
 = d\begin{pmatrix} \bar{p}+\bar{q} & \bar{p} \\ \bar{p} & \bar{q} \end{pmatrix}. \]
In particular, the matric Massey product is defined. One element of the Massey product is given by the class of 
\begin{multline*}
 \begin{pmatrix} \bar{p}+\bar{q} & \bar{p} \\ \bar{p} & \bar{q} \end{pmatrix}
 \cdot  \begin{pmatrix} \bar{y} & \bar{x}+\bar{y} \\ \bar{x} & \bar{y} \end{pmatrix}  
  +\begin{pmatrix} \bar{y} & \bar{x}+\bar{y} \\ \bar{x} & \bar{y} \end{pmatrix}\cdot \begin{pmatrix} \bar{p}+\bar{q} & \bar{p} \\ \bar{p} & \bar{q} \end{pmatrix} \\
  = \begin{pmatrix} \bar{p}\bar{y}+\bar{q}\bar{y}+\bar{p}\bar{x}+\bar{y}\bar{q}+\bar{x}\bar{p} & \bar{p}\bar{x}+\bar{q}\bar{x}+\bar{q}\bar{y}+\bar{y}\bar{p}+\bar{x}\bar{q}+\bar{y}\bar{q} \\ \bar{p}\bar{y}+\bar{q}\bar{x}+\bar{x}\bar{p}+\bar{x}\bar{q}+\bar{y}\bar{p} & \bar{p}\bar{x}+\bar{p}\bar{y}+\bar{q}\bar{y}+\bar{x}\bar{p}+\bar{y}\bar{q} \end{pmatrix}.
\end{multline*}
By Proposition~\ref{pqreigenschaften} the class of this matrix is $B=\smatrix{x^2+y^2 & 0 \\ x^2+y^2 & x^2+y^2}$. Let $C=\smatrix{y & x+y \\ x & y }$. Assume that $B$ lies in the indeterminacy; then there are $2\times 2$-matrices $Q$ and $R$ with $B=C Q+R  C$. Define $D=\smatrix{x & y \\ x+y & x}$; then $CD=DC=0$. If we denote by $\tr$ the trace of a matrix, then we have
\[ \tr(B D)=\tr(C Q D)+\tr(R C  D)
 =\tr(Q  D C)+\tr(R C  D)=0 \]
(note that these computations take place in a commutative ring). But
\[ \tr(B D)=\tr \begin{pmatrix} x^2y & \ast \\ \ast & 0 \end{pmatrix} = x^2y\neq 0, \]
a contradiction.\qed

In order to construct a module which is not a direct summand of a realisable one, it is often enough to consider 'ordinary' Massey products, i.e.~the case of $1\times 1$-matrices; this is true for example in the cases $G=\mathbb{Z}/2\mathbb{Z}\times\mathbb{Z}/2\mathbb{Z}$ (\cite{bks}, Example 7.7) and $G=\mathbb{Z}/3\mathbb{Z}$ (characteristic $3$, \cite{bks}, Example 7.6). In our present case, it is not that easy:
\begin{lemma} Let $k=\mathbb{F}_2$ be the field with $2$ elements. 
For all $a,b,c\in\Tate_L^*(k,k)$ satisfying $ab=0$ and $bc=0$ we have $0\in\left<a,b,c\right>$.
\end{lemma}
\beweis
By \cite{bks}, Lemma~5.14, the class  $m(a,b,c)$ is contained in the Massey product $\left<a,b,c\right>$. Therefore, it is enough to show that $m(a,b,c)$ is an element of the indeterminacy 
\[ a\cdot\Tate_L^{|b|+|c|-1}(k,k)+\Tate_L^{|a|+|b|-1}(k,k)\cdot c \]
for all $a,b,c$. By construction of $m$ it is enough to do so for those triples $(a,b,c)$ and $(sa,b,c)$ with $a,b,c\in\left\{1,x,y,x+y,x^2,y^2,x^2+y^2,x^2y\right\}$ which satisfy $ab=0$ and $bc=0$.

We have that $m(a,b,c)=0$: If $|a|,|b|\leq 1$, then $ab=0$ implies $a=0$ or $b=0$ (here we use that $k=\mathbb{F}_2$). If $|a|\geq 2$ or $|b|\geq 2$, then we get $m(a,b,c)=0$ from Theorem~\ref{quatemaintheorem}.

For $m(sa,b,c)$ we have by \eqref{mgleichung}
\[
 m(sa,b,c)=a\, m(s,b,c)+m(a,b,c)s.
\]
We have already seen that the second summand vanishes; the first summand is contained in
\begin{align*}
 a\cdot \Tate_L^{|s|+|b|+|c|-1}(k,k)=sa\cdot \Tate_L^{|b|+|c|-1}.\tag*{\qed}
\end{align*}

\bemerkung Note that the Proposition is not true for arbitrary fields of characteristic $2$: If the field $k$ contains an element $\alpha\in k$ satisfying $\alpha^2+\alpha+1=0$, then the Massey product
\[ \left< \alpha x+y,\alpha^2 x+y,\alpha x+y \right> \]
is defined and does not contain $0$.
 
\section{The case {$G=\left(\mathbb{Z}/2\mathbb{Z}\right)^r$}} \label{sfallzr}

Let us consider the case of (finite) abelian $p$-groups, which are of the form
\[ G=\prod_{i=1}^r \mathbb{Z}/p^{p_i}\mathbb{Z} \]
with $p_1,\dots,p_r\geq 1$. Note that, given groups $H_1,H_2$, there seems to be no obvious relation between the canonical classes $\gamma_{H_1}$, $\gamma_{H_2}$ and $\gamma_{H_1\times H_2}$. For example, the class $\gamma_{\mathbb{Z}/2\mathbb{Z}}$ is trivial, but (as noted in \cite{bks}, Example~7.7) $\gamma_{\mathbb{Z}/2\mathbb{Z} \times \mathbb{Z}/2\mathbb{Z}}$ is non-trivial. Our main result is the following.

\begin{satz} \label{rsatzallgemein}
Let $k$ be a field of characteristic $p>0$, and $G$ be the abelian $p$-group
\[ G=\prod_{i=1}^r \mathbb{Z}/m_i\mathbb{Z} \qquad \text{with $m_i=p^{p_i},\; p_i\geq 1$ for all $i=1,2,\dots,r$}. \]
\begin{itemize}
\item[(a)] Suppose $r=1$ or $r\geq 3$, and assume that $p^{p_i}\neq 3$ for all $i$. Then $\gamma_G=0\in \hochschild^{3,-1} \hat{H}^*(G,k)$. In particular, every $\hat{H}^*(G,k)$-module is a direct summand of a realisable module.
\item[(b)] In all other cases, $\gamma_G\neq 0\in \hochschild^{3,-1}\hat{H}^*(G,k)$, and there is an $\hat{H}^*(G,k)$-module which is not a direct summand of a realisable module.
\end{itemize}
\end{satz}
We will prove this in §~\ref{egeneralgroups}. The proof of (a) is rather complicated; for this reason we restrict ourselves to an easier special case in this section. 

\begin{satz} \label{rsatz} Let $k$ be a field of characteristic $2$ and $G=\left(\mathbb{Z}/2\mathbb{Z}\right)^r$ with $r\geq 3$. Then $\gamma_G=0$.
\end{satz}

\subsection{Projective resolution and cocycles} \label{sprojauflukozykel}
As a first step in the proof of Theorem~\ref{rsatz} , we will construct a complete projective resolution and corresponding cocycles. Let $r\geq 2$. We need some more new notation. In what follows, we will denote multi-indices with $r$ entries (i.e.,~elements of $\mathbb{Z}^r$) by Greek letters. For such an $\alpha$ let $\alpha_i$ denote its $i$-th entry. For $n\in\mathbb{Z}$ define
\begin{align*}
 M_n&=\left\{ \alpha\in\mathbb{Z}_{\geq 0}^r \mid \sum_{i=1}^r \alpha_i=n \right\}, \\
 N_n&=\left\{ \alpha\in\mathbb{Z}_{<0}^r \mid \sum_{i=1}^r \alpha_i=n-r+1 \right\}.
\end{align*}
For $\alpha\in M_n$ and $\beta\in M_m$ we have $\alpha+\beta\in M_{n+m}$, defined in the obvious way. Similarly $\alpha-\beta$ is an element of $\mathbb{Z}^r$; if 
$\alpha\geq \beta$ (i.e.~$\alpha_i\geq \beta_i$ for all $i$), then $\alpha-\beta\in M_{n-m}$.
For all $i$ we have an element $\varepsilon^i\in M_1$ with 
\[ \varepsilon^i_j = \begin{cases} 1 & \text{if $i=j$} \\0 & \text{otherwise.} \end{cases}  \]
Let us denote by $\svec{0}$ and $\svec{-1}$ the (unique) elements of $M_0$ and $N_{-1}$, respectively. 

Let $k$ be a field of characteristic $2$ and $G=\left(\mathbb{Z}/2\mathbb{Z}\right)^r$ with $r\geq 2$. Then 
\[ kG\cong L=k[z_1,z_2,\dots,z_r]/(z_i^2=0), \]
and as in §~\ref{szyklischegruppen} we will work with $L$ instead of $kG$. 
In $L$, we have the norm element $\norm=z_1z_2\cdots z_r$ and we define $\norm_i=z_1z_2\cdots\widehat{z_i}\cdots z_r$ for all $i=1,2,\dots,r$. Then,
\begin{align} \label{erechenregel}
 z_j\cdot \norm_i=\begin{cases} \norm&\textrm{if $j=i$} \\ 0&\textrm{if $j\neq i$.} \end{cases}
\end{align}
For every finite set $S$ we write $L^S$ for the free $L$-module with $L$-basis elements $(s)$ where $s$ runs through all elements of $S$. In particular,
\begin{align*}
 L^{M_n}&=\bigoplus_{\alpha\in M_n} L\cdot (\alpha), &
 L^{N_n}&=\bigoplus_{\alpha\in N_n} L\cdot (\alpha).
\end{align*}
In order to simplify notation, we define for $\alpha\in\mathbb{Z}^r$ with $\sum_i \alpha=n$
\begin{align} \label{esimplifynotation}
 (\alpha)=0 \in L^{M_n} \quad \textrm{if $\alpha_i<0$ for some $i$.}
\end{align}
Finally, put
$(\alpha)^*=(\svec{-1}-\alpha)\in L^{N_{-1-n}}$ for $\alpha\in M_n$.

A complete projective resolution of $k$ as a trivial $L$-module can be described as follows: Set $\hat{P}_n=L^{M_n}$ for $n\geq 0$ and $\hat{P}_n=L^{N_n}$ for $n<0$. The differential $\partial$ is defined to be
\begin{align*}
 \hat{P}_{n+1}=L^{M_{n+1}}&\longrightarrow L^{M_n}=\hat{P}_n\;:& (\alpha)&\mapsto \sum_{i=1}^r z_i(\alpha-\varepsilon^i) &&\text{for all $n\geq 0$} \\
 \hat{P}_{n+1}=L^{N_{n+1}}&\longrightarrow L^{N_n}=\hat{P}_n\;:& (\alpha)&\mapsto \sum_{i=1}^r z_i(\alpha-\varepsilon^i) &&\text{for all $n<-1$} \\
 \hat{P}_0=L^{M_0}&\longrightarrow L^{N_{-1}}=\hat{P}_{-1}\;: &(\svec{0})&\mapsto \norm\cdot(\svec{-1}) 
\end{align*}
We get a complex
\[ \dots\longleftarrow L^{N_{-2}}\longleftarrow L^{N_{-1}}\longleftarrow L^{M_0}\longleftarrow
 L^{M_1}\longleftarrow L^{M_2}\longleftarrow \dots\longleftarrow L^{M_n}\longleftarrow L^{M_{n+1}}\longleftarrow\dots \]
which is a resolution of $k$ as a trivial $L$-module, as we will see in §~\ref{allgemeinetheorie} in a more general context. In fact, it is a minimal resolution; after applying $\Hom_L(\_,k)$ we get the sequence
\[ \dots\stackrel{0}{\longrightarrow} k^{N_{-2}}
 \stackrel{0}{\longrightarrow} k^{N_{-1}} \stackrel{0}{\longrightarrow} k^{M_0}
 \stackrel{0}{\longrightarrow} k^{M_1} \stackrel{0}{\longrightarrow} k^{M_2}
 \stackrel{0}{\longrightarrow} \dots \stackrel{0}{\longrightarrow} k^{M_n} 
 \stackrel{0}{\longrightarrow} k^{M_{n+1}} \stackrel{0}{\longrightarrow} \dots \]
Hence
\[ \Tate^n_L(k,k)\cong \begin{cases} k^{M_n} & \textrm{if $n\geq 0$}
       \\ k^{N_n} & \textrm{if $n<0$} \end{cases} \]

Next, we will determine the multiplicative structure by explicit construction of cocycles of the endomorphism algebra of $\hat{P}_*$. We begin with degree $1$ and denote by $\bar{u}_i$ (for every $i=1,2,\dots,r$) the following morphism of chain complexes:
\begin{align*}
 \hat{P}_{n+1}=L^{M_{n+1}}&\longrightarrow L^{M_n}=\hat{P}_n\;:& (\alpha)&\mapsto\; (\alpha-\varepsilon^i) &&\text{for $n\geq 0$} \\
 \hat{P}_{n+1}=L^{N_{n+1}}&\longrightarrow L^{N_n}=\hat{P}_n\;:& (\alpha)&\mapsto\; (\alpha-\varepsilon^i) &&\text{for $n<-1$} \\
 \hat{P}_0=L^{M_0}&\longrightarrow L^{N_{-1}}=\hat{P}_{-1}\;: &(\svec{0})&\mapsto\; \norm_i\cdot (\svec{-1}) 
\end{align*}
To prove $\partial \bar{u}_i=\bar{u}_i\partial$, note that the diagram
\begin{align*}
\xymatrix{
\hat{P}_{n+1}\ar[d]^{\bar{u}_i}\ar[r]^-{\partial} & \hat{P}_{n}\ar[d]^{\bar{u}_i}  \\
\hat{P}_{n}\ar[r]_-{\partial} & \hat{P}_{n-1} }
&&
\xymatrix{
(\alpha) \ar@{|->}[r]\ar@{|->}[d] & \sum_{j=1}^r z_j(\alpha-\varepsilon^j)\ar@{|->}[d] \\
(\alpha-\varepsilon^i)\ar@{|->}[r] & \sum_{j=1}^r z_j(\alpha-\varepsilon^j-\varepsilon^i)
}
\end{align*}
commutes for $n\geq 1$, $\alpha\in M_{n+1}$ and for $n<-1$, $\alpha\in N_{n+1}$. Then we are left with two special cases, namely
\begin{align*}
\xymatrix{
(\varepsilon^m)\ar@{|->}[r]^{\partial}\ar@{|->}[d]^{\bar{u}_i} & z_m(\svec{0})\ar@{|->}[d]^{\bar{u}_i} \\
(\varepsilon^m-\varepsilon^i)\ar@{|->}[r]_-{\partial} & z_m\norm_i(\svec{-1}) }
&&
\xymatrix{
(\svec{0}) \ar@{|->}[r]^{\partial}\ar@{|->}[d]^{\bar{u}_i} & \norm(\svec{-1}) \ar@{|->}[d]^{\bar{u}_i} \\
\norm_i(\svec{-1})\ar@{|->}[r]_-{\partial}  & \norm(\svec{-1}-\varepsilon^i) }
\end{align*}
The left diagram shows the map $\hat{P}_1\longrightarrow \hat{P}_{-1}$. If $m\neq i$, then both compositions are $0$ due to \eqref{esimplifynotation} and \eqref{erechenregel}. If $m=i$ then we get $(\varepsilon^i)\mapsto \norm(\svec{-1})$ in both cases. The diagram on the right commutes because of the equality $\partial(\norm_i(\svec{-1}))=\sum_{j=1}^r z_j \norm_i(\svec{-1}-\varepsilon^j) =\norm(\svec{-1}-\varepsilon^i)$ which follows from \eqref{erechenregel}. We have thereby shown that $\bar{u}_i$ is a chain map; denote the corresponding cohomology class by $u_i$. One readily verifies that
\[ \Ext^*_L(k,k)\cong k[u_1,u_2,\dots,u_r]. \]
Since $u_iu_j=u_ju_i$, we know that $\bar{u}_i\bar{u}_j+\bar{u}_j\bar{u}_i$ is null-homotopic. In fact, it is almost zero: The cocycle $\bar{u}_i\bar{u}_j$ is given by
\begin{align*}
\hat{P}_{n+2}=L^{M_{n+2}}&\longrightarrow L^{M_{n}}=\hat{P}_{n}\;:&
 (\alpha)\mapsto&\; (\alpha-\varepsilon^j-\varepsilon^i) &&\text{for $n\geq 0$} \\
\hat{P}_1=L^{M_1}&\longrightarrow L^{N_{-1}}=\hat{P}_{-1}\;:& 
 (\varepsilon^m)\mapsto&\; \norm_i(\svec{-1}+\varepsilon^m-\varepsilon^j) \\
\hat{P}_0=L^{M_0}&\longrightarrow L^{N_{-2}}=\hat{P}_{-2}\;: &
 (\svec{0})\mapsto&\; \norm_j(\svec{-1}-\varepsilon^i) \\
\hat{P}_{n+2}=L^{N_{n+2}}&\longrightarrow L^{N_n}=\hat{P}_n\;:&
 (\alpha)\mapsto&\; (\alpha-\varepsilon^j-\varepsilon^i)  &&\text{for $n<-2$}
\end{align*}
By interchanging $i$ and $j$ one sees that $\bar{u}_i\bar{u}_j+\bar{u}_j\bar{u}_i$ is zero except in degrees $-1$ and $-2$, where it is given by
\[
\xymatrix{
\dots & \hat{P}_{0}\ar[l]\ar@{-->}[dr] \ar[d]_{\displaystyle (\svec{0})\mapsto \norm_i(\svec{-1}-\varepsilon^j)+\norm_j(\svec{-1}-\varepsilon^i)} &  \hat{P}_1 \ar[l]\ar[d]^{\displaystyle (\alpha)\mapsto \norm_i(\svec{-1}+\alpha-\varepsilon^j)+\norm_j(\svec{-1}+\alpha-\varepsilon^i)} &  \dots\ar[l] \\
\dots & \hat{P}_{-2}\ar[l] & \hat{P}_{-1} \ar[l] & \dots\ar[l] }
\]
As a null-homotopy for this we can take the map which is given by
\begin{equation} \label{xyhomotopy}
 \hat{P}_0 \ni (\svec{0}) \mapsto z_1\dots \widehat{z_i}\dots \widehat{z_j} \dots z_r (\svec{-1})\in \hat{P}_{-1} 
\end{equation}
and vanishes everywhere else.\bigskip

Now we consider the negative range. Take some multi-index $\beta\in\mathbb{Z}_{\geq 0}^r$ and write $m=|\beta|=\sum_i \beta_i$. Let us define a cocycle $\bar{\varphi}_\beta$ of degree $-(m+1)$ as follows:
\begin{align*}
\intertext{In degree $m+1+n$ with $n\geq 0$:}
  \hat{P}_{n}=L^{M_n} &\longrightarrow  L^{M_{m+1+n}}=\hat{P}_{m+1+n} \\
  (n\varepsilon^1) &\mapsto  \norm_1(\beta+(n+1)\varepsilon^1) \\
  (\alpha) &\mapsto 0 \quad\textrm{for all other $\alpha\in M_n$} \\
\intertext{In degree $n=0,1,\dots,m$:}
  \hat{P}_{-(m+1)+n}=L^{N_{-(m+1)+n}}&\longrightarrow L^{M_n}=\hat{P}_{n} \\
  (\svec{-1}-\alpha)=(\alpha)^* &\mapsto (\beta-\alpha) \quad\textrm{for all $\alpha\in M_{m-n}$} \\
\intertext{In degree $-n-1$ with $n\geq 0$:} 
  \hat{P}_{-(m+1)-n-1}=L^{N_{-(m+1)-n-1}}&\longrightarrow L^{N_{-n-1}}=\hat{P}_{-n-1} \\  
  (\beta+(n+1)\varepsilon^1)^*&\mapsto \norm_1(n\varepsilon^1)^* \\
  (\alpha)^* &\mapsto 0 \quad\textrm{for all other $\alpha\in M_{m+n+1}$}   
\end{align*}
Note that the first factor of $\left(\mathbb{Z}/2\mathbb{Z}\right)^r$ plays a special role. This is an arbitrary but, as it seems, unavoidable choice. We omit the straightforward proof of $\bar{\varphi}_\beta$ being a cocycle.

Note that in degree $0$ the map $\bar{\varphi}_\beta:L^{N_{-1-m}}\longrightarrow L$ is the projection onto the factor corresponding to $(\beta)^*$. Using \eqref{tateiso} we see that, for fixed $m$, the cohomology classes $\varphi_\beta$ with $|\beta|=m$ span $\smash{\Tate}^{-m-1}_L(k,k)$ as a $k$-vector space. We will see later that the Tate algebra is given by
\[ \Tate^*_L(k,k)\cong k[u_i,\varphi_\beta]\; / \sim \]
where $i=1,\dots,r$ and $\beta$ runs through all multi-indices; the relations are given by
$\varphi_\alpha\varphi_\beta=0$ for all $\alpha$ and $\beta$, and
\[ u_i\;\varphi_\beta=\begin{cases} 0 & \textrm{if $\beta_i=0$} \\ \varphi_{\beta-\varepsilon^i} & \textrm{otherwise.} \end{cases} \]
To prove these relations and for the computation of $\gamma$ we are supposed to write down several homotopies explicitly. We will circumvent this by proving that our projective resolution has a special lifting property which shows the existence of sufficiently nice homotopies in certain cases.

\subsection{The lifting property of the resolution} \label{sliftungseigen1}
Recall the definition of the map $\class$ of §~\ref{sklasse}. All properties carry over to our present case (because we are dealing with a minimal projective resolution again). 

\begin{definition} Let $J$ be an ideal of $L$. A map $f:P\rightarrow Q$ of $L$-modules is called a \emph{$J$-map}, if the image of $f$ is contained in the submodule $J\cdot Q\subseteq Q$. 
\end{definition}
The composition of a $J$-map with any map of $L$-modules is again a $J$-map. If $P$ and $Q$ are finitely generated free $L$-modules, we can think of $f$ as a matrix with coefficients in $L$. Then $f$ is a $J$-map if and only if all entries of this matrix lie in $J$. 
\begin{definition}
A map of chain complexes $f$ is called a $J$-map if $f_i$ is a $J$-map for all $i\in\mathbb{Z}$.
\end{definition}
The $J$-maps form a two-sided ideal of the endomorphism algebra.

Now let $I=\ker \epsilon=\left<z_1,\dots,z_r\right>_L$ be the augmentation ideal.
\begin{lemma} \label{nulllemma} Let $\bar{f}$ be a cocycle of the endomorphism algebra. Suppose $\bar{f}$ is an $I$-map. Then $\class(\bar{f})=0$. More generally, for every $I$-map $\bar{f}:\hat{P}[n]\longrightarrow \hat{P}$ we have $\class(\bar{f})=0$.
\end{lemma}
\beweis
$\class(\bar{f})$ is represented by $\hat{P}_n\xrightarrow{\bar{f}_0}\hat{P}_0\xrightarrow{\hspace{1pt}\epsilon} k,$
which is zero because the image of $\bar{f}_0$ is contained in the kernel of $\epsilon$.\qed
\begin{lemma} \label{ilemma} Suppose that, in the construction of a representative $m$ of $\gamma_G$, we are able to choose the map $f_2$ in such a way that $f_2(x,y)$ is an $I$-map for all $x,y\in \Tate^*_L(k,k)$. Then $m=0$, and hence $\gamma=0$.
\end{lemma}
\beweis
For all homogeneous $a,b,c\in\Tate^*_L(k,k)$ the value of $m(a,b,c)$ is the class of the cocycle
\[ -(-1)^{|a|} f_1(a)f_2(b,c)+f_2(ab,c)-f_2(a,bc)+f_2(a,b)f_1(c). \]
Since the $I$-maps form a two-sided ideal in the endomorphism algebra, this expression is an $I$-map. By Proposition~\ref{nulllemma} its class is zero, so $m(a,b,c)=0$.\qed\\
Our goal will be to choose $f_2$ in such a way that we can apply this proposition. Recall that, for $x,y\in\Tate^*_L(k,k)$, $f_2(x,y)$ has to be a homotopy for $f_1(xy)+f_1(x)f_1(y)$. We could write down all these homotopies and check that they are $I$-maps. This is a lot of work which we avoid (at least in negative degrees) by using the following lifting property of $\hat{P}_*$.
\begin{satz}[Lifting property of the resolution] \label{fsatz}
Let $f$ be a chain transformation of $\hat{P}$ of non-positive degree which is an $I^2$-map. Then there is a null-homotopy $h$ for $f$ which is an $I$-map.
\end{satz}
The proof is based on the following
\begin{lemma} \label{llemma}
Let $m,n$ be integers.
\begin{itemize}
\item[(i)] Let $g:\hat{P}_m\longrightarrow \hat{P}_n$ be an $I^2$-map satisfying $\partial\circ g=0$. If $n\geq 0$, then there is an $I$-map $h:\hat{P}_m\longrightarrow \hat{P}_{n+1}$, such that $g=\partial\circ h$.
\item[(ii)] Let $g:\hat{P}_m\longrightarrow \hat{P}_n$ be an $I^2$-map satisfying $g\circ \partial=0$. If $m<0$, then there is an $I$-map $h:\hat{P}_{m-1}\longrightarrow \hat{P}_n$, such that $g=h\circ\partial$.
\end{itemize}
\begin{align*} \xymatrix{
 & \hat{P}_{m}\ar[d]^{g}\ar@{-->}[dr]^{h} \\
 \hat{P}_{n-1} & \hat{P}_n\ar[l]^-\partial & \hat{P}_{n+1}\ar[l]^-\partial }  
&&
 \xymatrix{
 \hat{P}_{m-1}\ar@{-->}[dr]_h & \hat{P}_{m}\ar[l]^-\partial\ar[d]_g & \hat{P}_{m+1}\ar[l]^\partial \\
  & \hat{P}_n }
\end{align*}
\end{lemma}
\begin{proof}
Let us prove (i). Since $\hat{P}_n$ is projective and $\partial\circ g=0$, there is some $f:\hat{P}_m\longrightarrow \hat{P}_{n+1}$ satisfying $g=\partial\circ f$. Consider $f$ as a matrix with entries in $L$. Every element $z$ in $L$ can uniquely be written as $z=t+x$, where $t\in I$ and $x\in k\subset L$. Doing so for every coefficient of $f$, we get $f=h+X$ for some $I$-map $h$ and some matrix $X$ with coefficients in $k\subset L$. We claim that $\partial\circ X=0$ which implies that $h$ is a suitable lifting.

Since $h$ and $\partial$ are both $I$-maps, $\partial\circ h$ is an $I^2$-map. Thus,
\[ \partial\circ X=g-\partial\circ h \]
is an $I^2$-map. Let $V\subset L$ be the $k$-vector space spanned by $1,z_1,z_2,\dots,z_r$. All coefficients of $\partial$ lie in $V$. Since all coefficients of $X$ lie in $k$, we get that all coefficients of $\partial\circ X$ lie in $V$. From $V\cap I^2=\{0\}$ we get $\partial\circ X=0$. The proof of (ii) is similar. 
\end{proof}

\begin{proof}[Proof of Theorem~\ref{fsatz}]
Let $f:\hat{P}[n]\longrightarrow \hat{P}$ be a cocycle of non-positive degree $n$ and, at the same time, an $I^2$-map. We want to construct a null-homotopy $h$ for $f$ which is an $I$-map. 

Put $h_0=0$. We know that $\partial\circ f_0=0$, because $\partial=\partial_0$ is given by multiplication with the norm element $\norm$ and $\norm\cdot I^2=0$. From Proposition~\ref{llemma}.(i) we get some $I$-map $h_1:\hat{P}_n\longrightarrow \hat{P}_1$ satisfying $f_0=\partial\circ h_1$.

Suppose we have constructed $I$-maps $h_0,h_1,h_2,\dots,h_j$ in such a way that  $f_i=\partial h_{i+1}+h_i\partial$ for all $i=0,1,\dots,j-1$. Consider the $I^2$-map
$g=f_j-h_j\circ \partial$.
\[
\xymatrix@R=40pt@C=21pt{
\dots & \hat{P}_{n-1}\ar[l]\ar[d]_{f_{-1}}\ar@{-->}[dr]_{0} & \hat{P}_{n}\ar[l]\ar[d]_{f_{0}}\ar@{-->}[dr]_{h_1} &
\hat{P}_{n+1}\ar[l]\ar[d]_{f_1} & \dots\ar[l] & \hat{P}_{n+j-1}\ar[l]\ar[d]_{f_{j-1}}\ar@{-->}[dr]_{h_{j}} &
\hat{P}_{n+j}\ar[l]\ar[d]_{f_{j}}\ar@{-->}[dr]_{h_{j+1}} & \hat{P}_{n+j+1}\ar[l]\ar[d]_{f_{j+1}} & \dots\ar[l] \\
\dots & \hat{P}_{-1}\ar[l] & \hat{P}_{0}\ar[l] &
\hat{P}_{1}\ar[l] & \dots\ar[l] & \hat{P}_{j-1}\ar[l] &
\hat{P}_{j}\ar[l] & \hat{P}_{j+1}\ar[l] & \dots\ar[l] \\
}\]
Since
\[ \partial\circ g=\partial f_j-\partial h_j\partial=f_{j-1}\partial-(f_{j-1}-h_{j-1}\partial)\partial=0, \]
we get an $I$-map $h_{j+1}$ satisfying $f_j=\partial\circ h_{j+1}$ (by Proposition~\ref{llemma}.(i)). This defines $h$ in the positive range. Now,
\[ f_{-1}\circ \partial=\partial\circ f_0=0. \]
By Proposition~\ref{llemma}.(ii) there is an $I$-map $h_{-1}$ satisfying $f_{-1}=h_{-1}\circ \partial$. Continuing inductively using Proposition~\ref{llemma}.(ii) we end up with the desired null-homotopy $h$.
\end{proof}

\bemerkung
In general, the theorem is wrong if the degree of $f$ is positive. For Example, consider the following chain transformation of degree $1$:
\[
\xymatrix{
\dots &
\hat{P}_{-2} \ar[l] \ar[d]_0 &
\hat{P}_{-1} \ar[l]_\partial\ar[d]_0 &
\hat{P}_{0} \ar[l]_\norm \ar[d]_\norm&
\hat{P}_{1} \ar[l]_\partial \ar[d]_0&
\hat{P}_{2} \ar[l]_\partial \ar[d]_0&
\dots  \ar[l] \\
\dots &
\hat{P}_{-3} \ar[l] &
\hat{P}_{-2}  \ar[l]^\partial &
\hat{P}_{-1}  \ar[l]^\partial&
\hat{P}_{0}  \ar[l]^\norm&
\hat{P}_{1} \ar[l]^\partial&
\dots \ar[l]
} \]
Assume there is an $I$-homotopy $h$ for this cochain. Then we would have
$\norm\, h_0+h_1\, \norm = \norm$, but $\norm\cdot I=I\cdot \norm=0$, a contradiction.\bigskip

\bemerkung \label{iabbbemerkung}
The null-homotopy of $\bar{u}_i\bar{u}_j+\bar{u}_j\bar{u}_i$ given by \eqref{xyhomotopy} is an $I$-map if $r\geq 3$, because the only non-vanishing matrix coefficient is $z_1\dots\widehat{z_i}\dots\widehat{z_j}\dots z_r\in I$. Furthermore, the map $\bar{\varphi}_\beta$ is an $I^2$-map in degrees $>m$ and $<0$ (if $r\geq 3$), because the only non-vanishing coefficient is $\norm_1=z_2z_3\dots z_r$.

\subsection{End of the proof}
From now on, let $r\geq 3$. We will follow the construction of a representative $m$ for $\gamma$. At the same time we determine the multiplicative structure of the cohomology ring. Define the map $f_1$ on a $k$-basis of $\Tate^*_L(k,k)$ as follows:
\begin{align*}
 f_1\left(u_1^{\alpha_1}u_2^{\alpha_2}\cdots u_r^{\alpha_r}\right)&=
  \bar{u}_1^{\alpha_1}\bar{u}_2^{\alpha_2}\cdots \bar{u}_r^{\alpha_r} \\
 f_1\left(\varphi_\beta\right)&=\bar{\varphi}_\beta
\end{align*}
Here $\alpha$ and $\beta$ run through all (non-negative) multi-indices. Next, we are going to construct $f_2$ in the non-negative range. Instead of constructing $f_2$ explicitly, we will only prove the existence of a homotopy which is an $I$-map. 
\begin{lemma}\label{permutlemma0} Let $i_1,i_2,\dots,i_s$ be a sequence of indices from $\{1,2,\dots,r\}$, and let $\sigma$ be a permutation of $\{1,2,\dots,s\}$. Then the cocycle $\bar{u}_{i_1}\bar{u}_{i_2}\dots\bar{u}_{i_s}+\bar{u}_{i_{\sigma(1)}}\bar{u}_{i_{\sigma(2)}}\dots\bar{u}_{i_{\sigma(s)}}$ is null-homotopic via an $I$-map.
\end{lemma}
\begin{proof}
Let $K$ be the set of all permutations $\sigma$ such that for every sequence
$i_1,i_2,\dots,i_s$ of indices the cocycle
\[ \bar{u}_{i_1}\bar{u}_{i_2}\dots\bar{u}_{i_s}+\bar{u}_{i_{\sigma(1)}}\bar{u}_{i_{\sigma(2)}}\dots\bar{u}_{i_{\sigma(s)}} \]
is null-homotopic via an $I$-map. The set $K$ contains the trivial permutation and is closed under composition. Therefore, it is a subgroup of the symmetric group in $s$ letters.
On the other hand, $K$ contains all transpositions $(i,j)$: in this case we can construct a homotopy from the $\bar{u}_i$'s and the homotopy of $\bar{u}_i\bar{u}_j+\bar{u}_j\bar{u}_i$ given by~\eqref{xyhomotopy}, which is an $I$-map. Since the transpositions generate the symmetric group, we are done.
\end{proof}

For every multi-index $\alpha$ we will use the usual notation
$u^\alpha = u_1^{\alpha_1} u_2^{\alpha_2}\cdots u_r^{\alpha_r}.$
Let $\alpha$ and $\beta$ be given multi-indices. We have to choose for $f_2(u^\alpha,u^\beta)$ a null-homotopy of $f_1(u^{\alpha+\beta})+f_1(u^\alpha)f_1(u^\beta)$. But this cocycle is of the form described in Proposition~\ref{permutlemma0}; therefore, we can choose $f_2(u^\alpha,u^\beta)$ to be an $I$-map.

The next case is the one of two negative arguments, i.e.~we want to construct $f_2(\varphi_\alpha,\varphi_\beta)$. Note that $\bar{\varphi}_\alpha\bar{\varphi}_\beta$ is an $I^2$-map: 
\begin{align*}
 \xymatrix{
\dots & \hat{P}_{-(n+3)} \ar[l]\ar@{-->}[d] & \hat{P}_{-(n+2)} \ar[l]\ar@{-->}[d] & \hat{P}_{-(n+1)} \ar[l]\ar[d] & \hat{P}_{-n} \ar[l]\ar[d] & \dots\ar[l] & \ar@{}[d]|{\displaystyle \bar{\varphi}_\beta}  \\
\dots & \hat{P}_{-2} \ar[l]\ar[d] & \hat{P}_{-1} \ar[l]\ar[d] & \hat{P}_{0} \ar[l]\ar@{-->}[d] & \hat{P}_{1} \ar[l]\ar@{-->}[d] & \dots\ar[l] & \ar@{}[d]|{\displaystyle \bar{\varphi}_\alpha}  \\
\dots & \hat{P}_{m-1} \ar[l] & \hat{P}_{m} \ar[l] & \hat{P}_{m+1} \ar[l] & \hat{P}_{m+2} \ar[l] & \dots\ar[l] & }
\end{align*}
In this diagram we wrote $n=|\beta|$ and $m=|\alpha|$. The dashed arrows are $I^2$-maps (by Remark~\ref{iabbbemerkung}); hence, so is the composition. By Theorem \eqref{fsatz} there is a null-homotopy of $\bar{\varphi}_\alpha\bar{\varphi}_\beta$ which is an $I$-map. In particular, $\varphi_\alpha\varphi_\beta=0$ and $f_2(\varphi_\alpha,\varphi_\beta)$ can be chosen to be an $I$-map.\bigskip

As a last case consider $\bar{\varphi}_\beta\bar{u}_i$ for some index $i$ and a multi-index $\beta$. Let $n=|\beta|$. We get the following diagram:
\[
\xymatrix@C=15pt{
 &
\hat{P}_{-(n+2)} \ar[l]\ar[d]&
\hat{P}_{-(n+1)} \ar[l]\ar[d] &
\hat{P}_{-n}     \ar[l]\ar[d]&
\dots            \ar[l]\ar@{}[d]|{\displaystyle \bar{u}_i}&
\hat{P}_{-1}     \ar[l]\ar[d] & 
\hat{P}_{0}      \ar[l]\ar@{-->}[d] &
\hat{P}_{1}      \ar[l]\ar[d] &
\hat{P}_{2}      \ar[l]\ar[d]  &
\ar[l] \\
 &
\hat{P}_{-(n+3)} \ar[l]\ar@{-->}[d]&
\hat{P}_{-(n+2)} \ar[l]\ar@{-->}[d] &
\hat{P}_{-(n+1)} \ar[l]\ar[d]&
\dots            \ar[l]\ar@{}[d]|{\displaystyle \bar{\varphi}_\beta}&
\hat{P}_{-2}     \ar[l]\ar[d] & 
\hat{P}_{-1}     \ar[l]\ar[d] &
\hat{P}_{0}      \ar[l]\ar@{-->}[d] &
\hat{P}_{1}      \ar[l]\ar@{-->}[d]  &
\ar[l] \\
 &
\hat{P}_{-2}  \ar[l]&
\hat{P}_{-1}  \ar[l] &
\hat{P}_{0}   \ar[l]&
\dots         \ar[l]&
\hat{P}_{n-1} \ar[l] & 
\hat{P}_{n}   \ar[l] &
\hat{P}_{n+1} \ar[l] &
\hat{P}_{n+2} \ar[l]  &
\ar[l] \\
}
\]
Again, the dashed arrows are $I^2$-maps. This shows that $\bar{\varphi}_\beta\bar{u}_i$ consists of $I^2$-maps in degrees $\geq n$ and $\leq -1$. In the remaining degrees the map is given by
\begin{align}
 (\alpha)^*&=(\svec{-1}-\alpha)\mapsto (\beta-(\alpha+\varepsilon^i)) \label{bmapv}
\end{align}
Suppose $\beta_i>0$. Then the cocycle $\bar{\varphi}_{(\beta-\varepsilon^i)}$ is given by the same formula in degrees $0,1,\dots,n-1$; in the remaining degrees, it is an $I^2$-map. This implies that $\bar{\varphi}_{(\beta-\varepsilon^i)}-\bar{\varphi}_\beta\bar{u}_i$ is an $I^2$-map. In particular, $\varphi_{(\beta-\varepsilon^i)}=\varphi_\beta u_i$. If $\beta_i=0$ then the right hand side of \eqref{bmapv} is always zero; hence, $\bar{\varphi}_\beta\bar{u}_i$ is an $I^2$-map. Therefore, $\varphi_\beta u_i=0$. 

In both cases we can choose $f_2(\varphi_\beta,u_i)$ to be an $I$-map (using Theorem~\ref{fsatz}). We have already determined the multiplicative structure of $\Tate^*_L(k,k)$. Now we can extend the definition of $f_2$ to all pairs $(\varphi_\beta,u^\alpha)$ with multi-indices $\alpha$ and $\beta$ inductively: Suppose we have constructed $f_2$ for all pairs $(\varphi_\beta,u^\alpha)$ with $|\alpha|<m$, and let $\alpha$ be a multi-index with $|\alpha|=m$. Let $i$ be the maximal index satisfying $\alpha_i>0$. Then we have already chosen an $I$-map $f_2(\varphi_\beta,u^{\alpha-\varepsilon^i})$, and we can define
\[ f_2(\varphi_\beta,u^\alpha)=f_2(\varphi_\beta,u^{\alpha-\varepsilon^i})f_1(u_i)+f_2(\varphi_{\beta+\varepsilon^i-\alpha},u_i), \]
which is an $I$-map. Here we used the convention $\varphi_\delta=0$ if $\delta_i<0$ for some index $i$. Now we have defined an $I$-map $f_2(\varphi_\beta,u^\alpha)$ for all $\alpha,\beta$. The case $f_2(u^\alpha,\varphi_\beta)$ is similar.\bigskip

Now we have constructed $f_2$ in such a way that $f_2(x,y)$ is an $I$-map for all $x,y$. Using this, Theorem~\ref{rsatz} follows from Proposition~\ref{ilemma}.\qed

\section{Abelian $p$-groups}\label{egeneralgroups}

This section is devoted to the proof of Theorem~\ref{rsatzallgemein}. We will start in a more general context. Suppose we are given groups $G_1, G_2,\dots, G_r$ and corresponding complete projective resolutions $\hat{P}^\idx{i}$ of $k$ as a trivial $kG_i$-module. We will then show how to construct a complete projective resolution $\hat{P}$ of $k$ as a trivial $k(G_1\times G_2\times\dots\times G_r)$-module from these data. We will also construct elements of the endomorphism-dga of $\hat{P}$ from given elements of the endomorphism-dgas of the $\hat{P}^\idx{i}$.
Then we put $G_i=\mathbb{Z}/p^{p_i}\mathbb{Z}$ and generalise the lifting property of §~\ref{sliftungseigen1} to this case in §~\ref{eliftproperty}. In the last section we will use this to prove Theorem~\ref{rsatzallgemein}.

\subsection{General theory} \label{allgemeinetheorie}
Let $k$ be a field of characteristic $p>0$ and $r\geq 2$. Let $G_1,G_2,\dots,G_r$ be finite groups and $G=\prod_{i=1}^r G_i$. Then we have a canonical isomorphism $kG\cong kG_1\otimes kG_2\otimes\dots\otimes kG_r$ which we will suppress from notation. Given a $kG_i$-module $M_i$ for all $i=1,2,\dots,r$, the module $M_1\otimes\dots\otimes M_r$ gets a canonical $kG$-module structure. In particular, if $M_i=k$ is the trivial $kG_i$-module, then $M_1\otimes\dots\otimes M_r\cong k$ is the trivial $kG$-module.

As a first step, we construct a projective resolution (similar to \cite{Carlson}, Proposition~7.5). Suppose we are given complete projective resolutions $\hat{P}^\idx{i}$ of $k$ as a trivial $kG_i$-module for every $i=1,2,\dots,r$. By tensoring all the complexes $0\longleftarrow \hat{P}^\idx{i}_0\longleftarrow \hat{P}^\idx{i}_1\longleftarrow\dots$ over $k$ we get a complex
\begin{equation} \label{atpositiverkomplex}
 \hat{P}^+_*:\quad 0\longleftarrow \hat{P}_0\longleftarrow \hat{P}_1\longleftarrow \hat{P}_2\longleftarrow \dots 
\end{equation}
of $kG$-modules. The $n$-th module is given by
\[ \hat{P}_n=\bigoplus_{\alpha\in M_n} \hat{P}^\idx{1}_{\alpha_1}\otimes\dots\otimes \hat{P}^\idx{r}_{\alpha_r}, \]
where $M_n=\left\{ \alpha\in\mathbb{Z}_{\geq 0}^r \mid \sum_{i=1}^r \alpha_i=n \right\}$. By the Künneth theorem,
\begin{align}\label{cpposhomology}
H_j(\hat{P}^+)=\begin{cases} \bigotimes_{i=1}^r \hat{P}^\idx{i}_0 / \im \partial^\idx{i}_1 & \text{for $j=0$,} \\ 0 & \text{otherwise.} \end{cases} 
\end{align}
On the other hand we can tensor all the complexes $\dots\longleftarrow \hat{P}^\idx{i}_{-2}\longleftarrow \hat{P}^\idx{i}_{-1}\longleftarrow 0$ over $k$. Then we get a complex
\begin{align} 
 \hat{P}^-_*:\quad \dots \longleftarrow \hat{P}_{-3} \longleftarrow \hat{P}_{-2}\longleftarrow \hat{P}_{-1} \longleftarrow 0 
\end{align}
of $kG$-modules. Here
\[ \hat{P}_n=\bigoplus_{\alpha\in N_n} \hat{P}^\idx{1}_{\alpha_1}\otimes\dots\otimes \hat{P}^\idx{r}_{\alpha_r}, \]
where $N_n=\left\{ \alpha\in\mathbb{Z}_{<0}^r \mid \sum_{i=1}^r \alpha_i=n-r+1 \right\}$.
As before we get 
\begin{align} \label{cpneghomology} H_j(\hat{P}^-)=\begin{cases} \bigotimes_{i=1}^r \ker \partial^\idx{i}_{-1} & \text{for $j=-1$,} \\ 0 & \text{otherwise.} \end{cases}
\end{align}
By tensoring all the maps $\partial^\idx{i}_0:\hat{P}^\idx{i}_0\rightarrow \hat{P}^\idx{i}_{-1}$ we get a map
$\partial_0:\hat{P}_0\rightarrow \hat{P}_{-1}$ of $kG$-modules, which can be used to glue $\hat{P}^+$ and $\hat{P}^-$ together. We get a complex $\hat{P}$ of $kG$-modules with trivial homology (by \eqref{cpposhomology}, \eqref{cpneghomology} and the definition of $\partial_0$). Thus $\hat{P}$ is a complete resolution of $k$ as a trivial $kG$-module. The differential is given by the formula
\begin{align*}
\hat{P}_n &\longrightarrow \hat{P}_{n-1} \\
(a^\adx{1}, a^\adx{2},\dots, a^\adx{r}) &\mapsto \begin{cases} 
\displaystyle \sum_{|a^\adx{i}|>0}  (-1)^{|a^1|+\dots+|a^{i-1}|} (a^\adx{1}, a^\adx{2},\dots, \partial(a^\adx{i}),\dots, a^\adx{r}) & \text{if $n>0$} \\
\displaystyle \sum_{i} (-1)^{|a^1|+\dots+|a^{i-1}|+(i-1)}  (a^\adx{1}, a^\adx{2},\dots, \partial(a^\adx{i}),\dots, a^\adx{r}) & \text{if $n<0$} \\
(\partial(a^\adx{1}), \partial(a^\adx{2}),\dots,\partial(a^\adx{r})) & \text{if $n=0$} 
\end{cases}
\end{align*}
Denote by $A^\vdx{j}$ and $A$ the endomorphism dgas of $\hat{P}^\idx{j}$ and $\hat{P}$, respectively. We write $d$ for their differentials; these are defined as $df=\partial\circ f-(-1)^{|f|} f\circ \partial$. We are now going to construct elements in $A$ from given elements in the $A^\vdx{j}$'s. This will be done for elements of positive and negative degrees separately. We begin with the positive part. Let $f\in A^\vdx{j}$ be an element of degree $n\geq 0$. Define $\Phi(f)\in A$ of degree $n$ as follows: \label{phidefinition}%
\begin{align*}
\intertext{In degree $k\geq 0$:}
 \hat{P}_{k+n}&\longrightarrow \hat{P}_k \\
 (a^\idx{1}, \dots, a^\idx{r} ) &\mapsto \begin{cases} (-1)^{n(|a^1|+\dots+|a^{j-1}|)} (a^\idx{1},\dots,  a^\idx{j-1}, f(a^\idx{j}), a^\idx{j+1},\dots , a^\idx{r}) & \textrm{if $|a^\idx{j}|\geq n$} \\
 0 &\textrm{otherwise} \end{cases} 
\intertext{In degrees $k=-1,-2,\dots,-n$:}
 \hat{P}_{k+n}&\longrightarrow \hat{P}_k \\
 (a^\idx{1},\dots, a^\idx{r}) &\mapsto \begin{cases} (\partial a^\idx{1},\dots, \partial a^\idx{j-1}, f(a^\idx{j}), \partial a^\idx{j+1},\dots, \partial a^\idx{r}) & \textrm{if $|a^\idx{i}|=0$ for all $i\neq j$} \\
  0 & \textrm{otherwise} \end{cases}
\intertext{In degree $k<-n$:}
\hat{P}_{k+n}&\longrightarrow \hat{P}_k \\
(a^\idx{1}, \dots, a^\idx{r})&\mapsto  (-1)^{n(|a^1|+\dots+|a^{j-1}|+(j-1))} (a^\idx{1},\dots,  a^\idx{j-1}, f(a^\idx{j}), a^\idx{j+1},\dots , a^\idx{r}) 
\end{align*}
A straightforward calculation yields
\begin{lemma} \label{phidkomm} For every element $f\in A^\vdx{j}$ of non-negative degree we have  $\Phi(df)=d\Phi(f)$. \end{lemma}
\begin{korollar} \label{Phikorollar} The map $\Phi$ sends cocycles of $A^\vdx{j}$ of positive degree to cocycles of $A$. If the elements $f,f'\in A^\vdx{j}$ of positive degree are homotopic via $h\in A^\vdx{j}$, then $\Phi(f)$ and $\Phi(f')$ are homotopic via $\Phi(h)$.
\end{korollar}
\begin{lemma} \label{phimultiplikativ} For all $f,g\in A^\vdx{j}$ of positive degree we have $\Phi(fg)=\Phi(f)\Phi(g)$. \end{lemma}
We omit the straightforward proof.

Suppose we are given indices $j\neq l$ and cocycles $f\in A^\vdx{j}$ and $g\in A^\vdx{l}$ of positive degrees $n$ and $m$, respectively. Then we know that $\Phi(f)\Phi(g)$ and $(-1)^{nm}\Phi(g)\Phi(f)$ are homotopic cycles in $A$. In fact, we can easily write down a homotopy $h$ as follows: For $k=-1,-2,\dots,-n-m+1$ we define $h$ to be
\begin{align} \label{homotopieH}
 \hat{P}_{k+n+m-1}&\longrightarrow \hat{P}_{k} \\
 (a^\idx{1}, \dots, a^\idx{r})&\mapsto
  (-1)^{n+(m-1)|a^j|}(\partial a^\idx{1},\dots, f(a^j),\dots, g(a^l),\dots , \partial a^\idx{r}) 
\end{align}
whenever $|a^\idx{i}|=0$ for all $i\neq j,l$ and $|a^\adx{j}|<n$ and $|a^\adx{l}|<m$. In all other cases (and for all other values of $k$) put $h=0$. A somewhat lengthy calculation shows
\begin{lemma}\label{phifphigkomm} $\Phi(f)\Phi(g)-(-1)^{nm}\Phi(g)\Phi(f)=dh.$ \end{lemma}
Now we turn to elements of negative degree. In contrast to what we did before, we choose an element $f_j\in A^\vdx{j}$ of degree $n_j=|f_j|\leq 0$ for \textit{every} index $j$ in such a way, that at most one of them has degree $0$. From these data we construct an element $\Psi(f_1,f_2,\dots,f_r)$ of degree $n$ (where $n+1=(n_1+1)+\dots+(n_r+1)$) as follows:
\label{psidefinition}
\begin{align*} 
\intertext{In degree $k\geq -n$:}
 \hat{P}_{k+n}&\longrightarrow \hat{P}_k \\
 (a^\idx{1}, \dots, a^\idx{r}) &\mapsto \begin{cases} (-1)^{t+(n+n_1)|a^1|} (f_1(a^\idx{1}), f_2(\partial a^\adx{2}), f_3(\partial a^\adx{3}), \dots , f_r(\partial a^\idx{r}) ) & \textrm{if $|a^\idx{i}|=0$ for all $i\geq 2$} \\
 0 &\textrm{otherwise} \end{cases} 
\intertext{In degrees $k=0,1,\dots,-n-1$:}
 \hat{P}_{k+n}&\longrightarrow \hat{P}_k \\
 (a^\idx{1},\dots, a^\idx{r}) &\mapsto \begin{cases} (-1)^{\sum_{i=1}^r (n+n_1+\dots+n_i+i-1)|a^i|} (f_1(a^\idx{1}), f_2(a^\adx{2}),\dots, f_r(a^\idx{r})) & \textrm{if $|a^\idx{i}|\geq n_i$ for all $i$} \\
  0 & \textrm{otherwise} \end{cases}
\intertext{In degree $k<0$:}
\hat{P}_{k+n}&\longrightarrow \hat{P}_k \\
(a^\idx{1}, \dots, a^\idx{r}) &\mapsto \begin{cases} (-1)^{s+(n+n_1)|a^1|} (f_1(a^\idx{1}), \partial f_2(a^\adx{2}), \partial f_3(a^\adx{3}),  \dots, \partial f_r(a^\idx{r})) & \textrm{if $|a^\idx{i}|=n_i$ for all $i\geq 2$} \\
 0 &\textrm{otherwise} \end{cases} 
\end{align*}
where 
$t=\sum_{i=2}^r (n+n_1+\dots+n_i+i-1)$ and         
$s=\sum_{i=2}^r (n+n_1+\dots+n_i+i-1)n_i$.
This is a well-defined element of $A$. Note that the first factor of $G=\prod_{i=1}^r G_i$ plays a special role. This is an arbitrary but, as it seems, unavoidable choice. A direct computation shows
\begin{lemma} \label{dpsikomm} 
If the degree of $f_i$ is negative for all $i$, then
$\Psi(f_1,f_2,\dots,df_i,\dots,f_r)$ is defined for all $i$, and we have
\[ d\Psi(f_1,f_2,\dots,f_r)=\sum_{i=1}^r (-1)^{|f_1|+\dots+|f_{i-1}|} \Psi(f_1,f_2,\dots,df_i,\dots,f_r). \] 
\end{lemma}
As an immediate consequence we get
\begin{korollar} If the maps $f_1,\dots,f_r$ are cocycles of negative degree, then so is $\Psi(f_1,f_2,\dots,f_r)$. If in addition $f_j$ is null-homotopic via some $q_j$, then $\Psi(f_1,f_2,\dots,f_r)$ is null-homotopic via $\Psi(f_1,f_2,\dots,q_j,\dots,f_r)$. \end{korollar}

Let us investigate the relations of different compositions of $\Phi$'s and $\Psi$'s. To do so, we will use the notion of $I$-maps defined in §~\ref{sliftungseigen1}. 

From now on, let us assume that all the groups $G_j$ are commutative, and that $\hat{P}^\idx{j}_0=\hat{P}^\idx{j}_{-1}=kG_j$. Then we also have that $\hat{P}_0=\hat{P}_{-1}=kG$. Denote by $I^\vdx{j}$ the ideal $\im\partial^\idx{j}_0\cdot kG\subset kG$, and write $J_1$ for the ideal of $kG$ generated by all the ideals $I^\vdx{j}$, $j=1,2,\dots,r$. Finally, define $J_m$ to be the ideal $J_1^m$ for all $m$. Then we have a filtration 
\[ kG=:J_0\supset J_1\supset J_2\supset\dots\supset J_r \]
of ideals in $kG$. 

\begin{bemerkung}\label{phifbemerkung}
Recall the definition of $\Phi(f)$ for some $f\in A^\vdx{j}$ with $|f|=n$. In degree $k=-1,-2,\dots,-n$ it was given by
\begin{align*}
 \hat{P}_{k+n}&\longrightarrow \hat{P}_k \\
 (a^\idx{1},\dots,a^\idx{r}) &\mapsto \begin{cases} \pm(\partial a^\idx{1},\dots,\partial a^\idx{j-1},f(a^\idx{j}),\partial a^\idx{j+1},\dots,\partial a^\idx{r}) & \textrm{if $|a^\idx{i}|=0$ for all $i\neq j$} \\
  0 & \textrm{otherwise} \end{cases}
\end{align*}
On the right hand side we have $r-1$ differentials of elements of degree $0$. Therefore, these are $J_{r-1}$-maps. The same is true for $\Psi(f_1,\dots,f_r)$ in the 'exterior' range $k\geq -n$ and $k<0$. Furthermore, the homotopy $h$ given by \eqref{homotopieH} is a $J_{r-2}$-map.
\end{bemerkung}

\begin{lemma} \label{phikombo} Suppose we are given $f_i,g_i\in A^\vdx{i}$ of negative degree for all $i=1,2,\dots,r$. Then the composite $\Psi(f_1,f_2,\dots,f_r)\Psi(g_1,g_2,\dots,g_r)$ is a $J_{r-1}$-map. 
\end{lemma}
\beweis
Let us write $f=\Psi(f_1,f_2,\dots,f_r)$, $g=\Psi(g_1,g_2,\dots,g_r)$, $m=|f|$ and $n=|g|$. Consider the composition $f\circ g$:
\[ \xymatrix{
\dots & \hat{P}_{n-2} \ar[l]\ar@{-->}[d] & \hat{P}_{n-1} \ar[l]\ar@{-->}[d] & \hat{P}_{n} \ar[l]\ar[d] & \hat{P}_{n+1} \ar[l]\ar[d] & \dots\ar[l] & \ar@{}[d]|{\displaystyle f}  \\
\dots & \hat{P}_{-2} \ar[l]\ar[d] & \hat{P}_{-1} \ar[l]\ar[d] & \hat{P}_{0} \ar[l]\ar@{-->}[d] & \hat{P}_{1} \ar[l]\ar@{-->}[d] & \dots\ar[l] & \ar@{}[d]|{\displaystyle g}  \\
\dots & \hat{P}_{-m-2} \ar[l] & \hat{P}_{-m-1} \ar[l] & \hat{P}_{-m} \ar[l] & \hat{P}_{-m+1} \ar[l] & \dots\ar[l] & }
 \]
By definition, the dashed arrows are $J_{r-1}$-maps. Therefore, the composition is a $J_{r-1}$-map in every degree. \qed
\begin{lemma} \label{psiphikombo} Let $f_i\in A^\vdx{i}$ be of negative degree $n_i$ for all $i=1,2,\dots,r$, and let $g\in A^\vdx{j}$ be of non-negative degree $m$. Write $n=|\Psi(f_1,\dots,f_r)|$. 
\begin{itemize}
\item[(i)] If $|g\circ f_j|<0$, then
\[ \Phi(g)\circ \Psi(f_1,f_2,\dots,f_r)-(-1)^{m(n_1+\dots+n_{j-1})} \Psi(f_1,f_2,\dots,g\circ f_j,\dots,f_r) \]
is a $J_{r-1}$-map. If $|g\circ f_j|\geq 0$, then
$\Phi(g)\circ \Psi(f_1,f_2,\dots,f_r)$ is a $J_{r-1}$-map.
\item[(ii)] If $|f_j\circ g|<0$, then
\[ \Psi(f_1,f_2,\dots,f_r)\circ\Phi(g) -(-1)^{m(n+n_1+\dots+n_{j-1})} \Psi(f_1,f_2,\dots,f_j\circ g,\dots,f_r) \]
is a $J_{r-1}$-map. If $|f_j\circ g|\geq 0$, then
$\Psi(f_1,f_2,\dots,f_r)\circ \Phi(g)$ is a $J_{r-1}$-map.
\end{itemize}
\end{lemma}

\beweis
We will show (i); the proof of (ii) is similar. Let us write $\psi=\Psi(f_1,f_2,\dots,f_r)$, $n=|\psi|<0$ and $\Phi_g=\Phi(g)$. Suppose $m+n\geq 0$. Then we are in the following situation:
\[
\xymatrix@1@C=12pt{
\dots\ar@{-->}[d]  & \hat{P}_{n-1} \ar[l]\ar@{-->}[d] &  \hat{P}_n \ar[l]\ar[d] & 
\dots \ar@{}[d]|{\displaystyle\psi}\ar[l] & \hat{P}_{-1} \ar[l]\ar[d] & \hat{P}_0 \ar[l]\ar@{-->}[d] &
\dots \ar[l]\ar@{-->}[d] & \hat{P}_{m+n-1} \ar[l]\ar@{-->}[d] & \hat{P}_{m+n} \ar[l]\ar@{-->}[d] & \dots \ar[l]\ar@{-->}[d]  \\
\dots  & \hat{P}_{-1} \ar[l]\ar[d] &  \hat{P}_0 \ar[l]\ar@{-->}[d] & 
\dots \ar[l]\ar@{-->}[d] & \hat{P}_{-n-1} \ar[l]\ar@{-->}[d] & \hat{P}_{-n} \ar[l]\ar@{-->}[d] &
\dots \ar[l]\ar@{-->}[d] & \hat{P}_{m-1} \ar[l]\ar@{-->}[d] & \hat{P}_{m} \ar[l]\ar[d] & \dots \ar[l] \ar@{}[d]|{\displaystyle\Phi_g} \\
\dots  & \hat{P}_{-m-1} \ar[l] &  \hat{P}_{-m} \ar[l] & 
\dots \ar[l] & \hat{P}_{-n-m-1} \ar[l] & \hat{P}_{-n-m} \ar[l] &
\dots \ar[l] & \hat{P}_{-1} \ar[l] & \hat{P}_{0} \ar[l] & \dots \ar[l]  
} \]
The dashed arrows are $J_{r-1}$-maps, and hence so is the composition $\Phi_g\circ \psi$.
In the case $m+n<0$, we have the following diagram:
\[
\xymatrix@1@C=12pt{
\dots\ar@{-->}[d]  & \hat{P}_{n-1} \ar[l]\ar@{-->}[d] &  \hat{P}_n \ar[l]\ar[d] & 
\dots \ar@{}[d]|{\displaystyle\psi}\ar[l] & \hat{P}_{m+n-1} \ar[l]\ar[d] & \hat{P}_{m+n} \ar[l]\ar[d] &
\dots \ar[l] & \hat{P}_{-1} \ar[l]\ar[d] & \hat{P}_{0} \ar[l]\ar@{-->}[d] & \dots \ar[l]\ar@{-->}[d]  \\
\dots  & \hat{P}_{-1} \ar[l]\ar[d] &  \hat{P}_0 \ar[l]\ar@{-->}[d] & 
\dots \ar[l]\ar@{-->}[d] & \hat{P}_{m-1} \ar[l]\ar@{-->}[d] & \hat{P}_{m} \ar[l]\ar[d] &
\dots \ar[l] & \hat{P}_{-n-1} \ar[l]\ar[d] & \hat{P}_{-n} \ar[l]\ar[d] & \dots \ar[l] \ar@{}[d]|{\displaystyle\Phi_g} \\
\dots  & \hat{P}_{-m-1} \ar[l] &  \hat{P}_{-m} \ar[l] & 
\dots \ar[l] & \hat{P}_{-1} \ar[l] & \hat{P}_{0} \ar[l] &
\dots \ar[l] & \hat{P}_{-n-m-1} \ar[l] & \hat{P}_{-n-m} \ar[l] & \dots \ar[l]  
} \]
Again the dashed arrows are $J_{r-1}$-maps; in the remaining cases the composition is given by 
\begin{eqnarray}
(a^\idx{1},\dots, a^\idx{r}) &\stackrel{\psi}{\mapsto}&\begin{cases} \pm (f_1(a^\idx{1}),\dots, f_r(a^\idx{r})) & \text{if $|a^\idx{i}|\geq n_i$ for all $i$} \\ 0 & \text{otherwise} \end{cases} \nonumber \\ &\stackrel{\Phi_g}{\mapsto}&\begin{cases} \pm (f_1(a^\idx{1}),\dots, g(f_j(a^\adx{j})),\dots, f_r(a^\idx{r})) & \begin{minipage}[t]{4cm} if $|a^\idx{i}|\geq n_i$ for all $i$ \\ \hspace*{0.1cm} and $|a^\adx{j}|\geq m+n_j$ \end{minipage} \\ 0 & \text{otherwise.} \end{cases} \label{phipsi1}
\end{eqnarray}
Here, the signs can be deduced from the definition of $\Phi$ and $\Psi$. Now consider the two cases given in (i). If $|g\circ f_j|<0$, then we have $m+n<0$ and we are in the situation of the second diagram. The maps $\Phi_g\circ \psi$ and $\Psi(f_1,f_2,\dots,g\circ f_j,\dots,f_r)$ agree in the range described by \eqref{phipsi1}. Outside this range, they are both $J_{r-1}$-maps. Therefore, their difference must be a $J_{r-1}$-map everywhere, which proves the proposition in this case.

Now suppose $|g\circ f_j|\geq 0$, i.e.~$m+n_j\geq 0$. If $m+n\geq 0$, then the first diagram applies; as we have seen already, the composition $\Phi_g\circ \psi$ is a $J_{r-1}$-map. We are left with the case $m+n<0$. Then the first case in \eqref{phipsi1} cannot occur (because $|a^\adx{j}|<0$); hence, the formula \eqref{phipsi1} always gives $0$. Then the second diagram shows that the composition is a $J_{r-1}$-map.\qed

\subsection{Abelian $p$-groups}
In this section we consider arbitrary (finite) abelian $p$-groups. Let $r\geq 2$ and
\[ G=\prod_{i=1}^r \mathbb{Z}/m_i\mathbb{Z} \qquad \text{where $m_i=p^{p_i},\; p_i\geq 1$ for $i=1,2,\dots,r$}. \]
Let $G_i=\mathbb{Z}/m_i\mathbb{Z}$. As in §~\ref{szyklischegruppen}, we have a complete projective resolution $\hat{P}^i_*$ of the trivial $kG_i$-module $k$ together with cocycles $\Bar{x}_i$ and $\Bar{y}_i$ of the endomorphism dga $A^\vdx{i}$. These cocycles satisfy $\Bar{x}_i^2=\Bar{y}_i$ if $m_i=2$, and if $m_i\geq 3$ we know that $\Bar{x}_i^2$ is null-homotopic via $\Bar{q}_i$. 

Using the construction of §~\ref{allgemeinetheorie} we get a complete resolution $\hat{P}_*$ of $k$ as a trivial $kG$-module. This resolution is not as complicated as it looks. Let us give a more direct description of $\hat{P}_*$. Note that the isomorphism
\begin{equation} \label{algebreniso1} kG_1\otimes\dots\otimes kG_r\cong kG, \end{equation}
 is compatible with the isomorphisms $kG_i\cong k[z_i]/(z_i^{m_i})$ and $kG\cong k[z_1,\dots,z_r]/(z_i^{m_i})_i$. This allows us to consider the elements of $kG_i$ as elements of $kG$. For every $\alpha\in M_n$ (with $n\geq 0$) we have a direct summand 
\[ \hat{P}^1_{\alpha_j}\otimes\dots\otimes \hat{P}^r_{\alpha_r}=kG_1\otimes\dots\otimes kG_r\cong kG \]
of $\hat{P}_n$; as in §~\ref{sfallzr}, we denote by $(\alpha)$ the corresponding $kG$-basis element of $\hat{P}_n$. Then we have isomorphisms $\hat{P}_n\cong \bigoplus_{\alpha\in M_n} kG(\alpha)$, which will be suppressed in our notation. Similarly one has $\hat{P}_{n}\cong \bigoplus_{\alpha\in N_n} kG(\alpha)$ for $n<0$. Using these isomorphisms, we consider $\hat{P}_*$ as a complex
\[ \dots\longleftarrow kG^{N_{-3}}\longleftarrow kG^{N_{-2}}\longleftarrow kG^{N_{-1}}\longleftarrow 
 kG^{M_0}\longleftarrow kG^{M_1}\longleftarrow kG^{M_2}\longleftarrow \dots \]
Now we are going to describe the differential. To do so, we need some more notation. If $\alpha\in M_n$ and $\beta\in M_m$ are such that $\alpha_i\geq \beta_i$ for all indices $i$, we have $\alpha-\beta\in M_{n-m}$ and $(\alpha-\beta)$ is some basis element of $\hat{P}_{n-m}$. On the other hand, if $\alpha_i<\beta_i$ for some index $i$, we define 
\[ (\alpha-\beta)=0\in\hat{P}_{n-m}. \]
As in §~\ref{sprojauflukozykel}, the generators of $\hat{P}_0$ and $\hat{P}_{-1}$ will be denoted by $(\svec{0})$ and $(\svec{-1})$, respectively.

For $n\in\mathbb{Z}$ and $i=1,2,\dots,r$ we define $[n]_i$ as follows:
\[ [n]_i=\begin{cases} 1 & \text{if $n$ is odd} \\ m_i-1 & \text{if $n$ is even} \end{cases} \]
Whenever the index $i$ is clear from context, we will drop it from the notation; in particular in the case of $[n]$ occurring in the exponent of $z_i$, or the argument of $[\cdot]$ being the $i$-th component of some multi-index, e.g.~$[\alpha_i+1]$. With this notation, the complex $\hat{P}^i$ is given by
\[ \dots\xleftarrow{-z_i^{[-2]}}\hat{P}^i_{-2}\xleftarrow{z_i^{[-1]}}\hat{P}^i_{-1}
 \xleftarrow{-z_i^{[0]}}\hat{P}^i_0\xleftarrow{z_i^{[1]}}\hat{P}^i_1\xleftarrow{-z_i^{[2]}}\dots
 \]
We are now able to describe the differential $\partial:\hat{P}_{k+1}\longrightarrow \hat{P}_k$ as follows:
\begin{align*}
(\alpha)&\mapsto \sum_{i=1}^r \pm z_i^{[\alpha_i]} (\alpha-\varepsilon^i) &&\text{for $k\geq 0$} \\
(\svec{0})&\mapsto z_1^{[0]}z_2^{[0]}\dots z_r^{[0]} (\svec{-1}) &&\text{for $k=-1$} \\
(\alpha)&\mapsto \sum_{i=1}^r \pm z_i^{[\alpha_i]} (\alpha+\varepsilon^i) &&\text{for $k<0$}
\end{align*}
We have the augmentation ideal \[ I=\ker\varepsilon=\ker\partial_0=\left<z_1,z_2,\dots,z_r\right>_{kG}\subseteq kG. \]
Note that $\im \partial_n\subset I\cdot P_{n-1}$,
which implies that $\Hom_{kG}(\partial,k)=0$; hence
\[ \Tate^n_{kG}(k,k)\cong \begin{cases} k^{M_n} & \text{for $n\geq 0$,} \\ k^{N_n} & \text{for $n<0$.} \end{cases} \]

\subsection{The lifting property}\label{eliftproperty}
In this section we generalise the result of §~\ref{sliftungseigen1} as follows:
\begin{satz} \label{flemma2}
Suppose we are given a cocycle $f:\hat{P}[n]\longrightarrow\hat{P}$ of non-positive degree which is a $J_2$-map. Then there is a null-homotopy $h$ of $f$. Furthermore, this $h$ can be chosen to be an $I$-map.
\end{satz}
Note that for every $m$, $J_m$ is the ideal generated by all the products of $m$ monomials of the form $z_i^{m_i-1}=z_i^{[0]}$.

Due to the more complicated form of the differentials, the proof of the theorem will be more elaborate than before. The \textit{idea} is as follows. Let $n$ be a given degree. We will construct splittings $\Hom_L(\hat{P}_{k+n},\hat{P}_k)\cong A_k\oplus B_k$ and $\Hom_L(\hat{P}_{k+n-1},\hat{P}_k)\cong C_k\oplus D_k$ as $k$-vector spaces in such a way, that the two maps
\begin{eqnarray*}
 \partial^*:\Hom_L(\hat{P}_{k+n-1},\hat{P}_k) &\longrightarrow& \Hom_L(\hat{P}_{k+n},\hat{P}_k) \\
 \phi&\mapsto& \phi\circ \partial \\
 \partial_*:\Hom_L(\hat{P}_{k+n},\hat{P}_{k+1}) &\longrightarrow& \Hom_L(\hat{P}_{k+n},\hat{P}_k) \\
 \phi&\mapsto& \partial\circ \phi 
\end{eqnarray*}
both respect the decompositions. Furthermore, this is done in such a way that $J_2$-maps of degree $n$ belong to $A_k$, and elements of $C_k$ are $I$-maps.

Given a $J_2$-map $f$ and a null-homotopy $h$ for $f$, we have
\[ f_k=\partial_* h_{k+1} - (-1)^{|h|} \partial^* h_k \]
This equation must still hold after restriction to $A_k$. Therefore, we can erase the $D_k$-part of $h$ to obtain an $I$-map, which is still a null-homotopy for $f$. This is the main idea. Nevertheless, the somewhat special $\partial:\hat{P}_0\longrightarrow \hat{P}_{-1}$ will make the proof more complicated.\bigskip

\beweis[ of Theorem~\ref{flemma2}]
\newcommand{\monobasis}{\mathcal{M}}
The $k$-vector space $kG$ has a basis $\monobasis$ consisting of all the monomials in the variables $z_1,\dots,z_r$. Define for all multi-indices $\alpha,\beta\in\mathbb{Z}^r$
\begin{align*} \mathcal{I}_{\alpha,\beta}&=\left\{ \prod_{s\in S} z_s^{[\alpha_s+1]-[\beta_s+1]} \,\Bigl\vert\, S\subseteq\{1,2,\dots,r\}\;\text{with $[\alpha_s+1]\geq [\beta_s+1]$ for all $s\in S$} \right\} 
\intertext{and}
 \mathcal{J}_{\alpha,\beta}&=\left\{ z_j^{\nu_j} z  \mid z\in\mathcal{I}_{\alpha,\beta},\;\text{$\nu_j\in\{[\alpha_j+1],[\beta_j]\}$ and $z_j$ is not a factor of $z$} \right\}. 
\end{align*}
Write $\tilde{\mathcal{I}}_{\alpha,\beta}=\monobasis\setminus\mathcal{I}_{\alpha,\beta}$ and $\tilde{\mathcal{J}}_{\alpha,\beta}=\monobasis\setminus \mathcal{J}_{\alpha,\beta}$. 
Every element $g\in\Hom_L(\hat{P}_{p},\hat{P}_{q})$ can be uniquely written as
\[ (\beta) \mapsto \sum_{\alpha} g_{\alpha,\beta}(\alpha). \]
Here $g_{\alpha,\beta}\in L$ are the coefficients of the matrix representation of $g$ with respect to our chosen bases. Let $n<0$. For a set $\mathcal{S}$, we write $\left<\mathcal{S}\right>_k$ for the $k$-vector space generated by the elements of $\mathcal{S}$. For all $k\in\mathbb{Z}$ define
\begin{align*}
 A_k&=\left\{ f\in\Hom_L(\hat{P}_{n+k},\hat{P}_k) \mid f_{\alpha,\beta}\in \bigl<\tilde{\mathcal{J}}_{\alpha,\beta}\bigr>_k\;\text{for all $\alpha,\beta$} \right\}, \\
 B_k&=\left\{ f\in\Hom_L(\hat{P}_{n+k},\hat{P}_k) \mid f_{\alpha,\beta}\in \bigl<\mathcal{J}_{\alpha,\beta}\bigr>_k\;\text{for all $\alpha,\beta$} \right\},  \\
 C_k&=\left\{ h\in\Hom_L(\hat{P}_{n+k-1},\hat{P}_k) \mid h_{\alpha,\beta}\in \bigl<\tilde{\mathcal{I}}_{\alpha,\beta}\bigr>_k\;\text{for all $\alpha,\beta$} \right\}, \\
 D_k&=\left\{ h\in\Hom_L(\hat{P}_{n+k-1},\hat{P}_k) \mid h_{\alpha,\beta}\in \bigl<\mathcal{I}_{\alpha,\beta}\bigr>_k\;\text{for all $\alpha,\beta$} \right\}. 
\end{align*}
That is, for $h\in C_k$ we require the matrix coefficients not to contain certain monomials in the $z_i$'s, namely those of $\mathcal{I}$. This set depends on the position in the matrix. Anyway, the monomial~$1$ is contained in $\mathcal{I}$; therefore all elements of $C_k$ are $I$-maps.\bigskip

We have decompositions of $k$-vector spaces
\[ \Hom_L(\hat{P}_{n+k},\hat{P}_k)=A_k\oplus B_k\quad\text{and}\quad
\Hom_L(\hat{P}_{n+k-1},\hat{P}_k)=C_k\oplus D_k. \]

Now one has the following fundamental
\begin{lemma} \label{partialzerlegung} The map
\[ \partial^*:C_k\oplus D_k=\Hom_L(\hat{P}_{k+n-1},\hat{P}_k)\longrightarrow \Hom_L(\hat{P}_{k+n},\hat{P}_k)=A_k\oplus B_k \]
respects the decomposition for $k\neq -n$. For $k=-n$ it maps $C_k$ to $0$.
The map
\[ \partial_*:C_{k+1}\oplus D_{k+1}=\Hom_L(\hat{P}_{k+n},\hat{P}_{k+1})\longrightarrow \Hom_L(\hat{P}_{k+n},\hat{P}_k)=A_k\oplus B_k \]
respects the decomposition for $k\neq -1$. For $k=-1$ it maps $C_{k+1}$ to $0$.
\end{lemma}
\beweis
First we claim for all $i$ and for all multi-indices $\alpha,\beta$:
\begin{itemize}
\item[(i)] $z_i^{[\beta_i]} \mathcal{I}_{\alpha,\beta-\varepsilon^i} 
              \subset \mathcal{J}_{\alpha,\beta}$
\item[(ii)] $z_i^{[\beta_i]} \tilde{\mathcal{I}}_{\alpha,\beta-\varepsilon^i} 
               \subset \tilde{\mathcal{J}}_{\alpha,\beta}\cup \{0\}$
\item[(iii)] $z_1^{[0]}z_2^{[0]}\dots z_r^{[0]} \mathcal{\tilde{I}}_{\alpha,\beta}=0$
\end{itemize}
Ad (i): Let $z\in\mathcal{I}_{\alpha,\beta-\varepsilon^i}$, then $z$ is of the form
\[ z=\prod_{s\in S} z_s^{[\alpha_s+1]-[(\beta-\varepsilon^i)_s+1]} \]
for some index set $S$. If $i\notin S$, then $z\in\mathcal{I}_{\alpha,\beta}$ and hence $z_i^{[\beta_i]}z\in \mathcal{J}_{\alpha,\beta}$. On the other hand, if $i\in S$, then
\[ z_i^{[\beta_i]}z=z_i^{[\beta_i]} z_i^{[\alpha_i+1]-[\beta_i]} \prod_{s\in S\setminus \{i\}}  z_s^{[\alpha_s+1]-[\beta_s+1]} 
=z_i^{[\alpha_i+1]}\prod_{s\in S\setminus \{i\}} z_s^{[\alpha_s+1]-[\beta_s+1]}
\in \mathcal{J}_{\alpha,\beta}. \]

Ad (ii): The set $\tilde{\mathcal{I}}_{\alpha,\beta-\varepsilon^i}$ consists of monomials in the $z_i$'s. Therefore the set $z_i^{[\beta_i]}\tilde{\mathcal{I}}_{\alpha,\beta-\varepsilon^i}$ also consists of monomials (and possibly $0$). Thus it is enough to prove that for every monomial $x$ satisfying
$z_i^{[\beta_i]}x\in\mathcal{J}_{\alpha,\beta}$ we have $x\in\mathcal{I}_{\alpha,\beta-\varepsilon^i}$. From $z_i^{[\beta_i]}x\in\mathcal{J}_{\alpha,\beta}$ we get
\[ z_i^{[\beta_i]}x=z_j^{\nu}\prod_{s\in S} z_s^{[\alpha_s+1]-[\beta_s+1]} \]
for some index set $S$, $j\notin S$ and $\nu\in\{[\alpha_j+1],[\beta_j]\}$. Since 
\[ [\beta_i]=m_i-[\beta_i+1]>[\alpha_i+1]-[\beta_i+1], \]
$i$ cannot belong to $S$; hence $i=j$.
If $\nu=[\beta_i]$ then
\[ x=\prod_{s\in S} z_s^{[\alpha_s+1]-[\beta_s+1]} \in\mathcal{I}_{\alpha,\beta-\varepsilon^i}. \]
If $\nu=[\alpha_i+1]$ then in particular $[\alpha_i+1]\geq [\beta_i]$ and
\[ x=z_i^{[\alpha_i+1]-[\beta_i]} \prod_{s\in S} z_s^{[\alpha_s+1]-[\beta_s+1]} \in\mathcal{I}_{\alpha,\beta-\varepsilon^i}. \]
Ad (iii): We have $1\in\mathcal{I}_{\alpha,\beta}$ for all $\alpha,\beta$. Hence
$1\notin\tilde{\mathcal{I}}_{\alpha,\beta}$. But all the other monomials in the $z_i$'s vanish when multiplied with $z_1^{[0]}z_2^{[0]}\dots z_r^{[0]}$.

Now we can prove the first part of the proposition using (i), (ii) and (iii) as follows: Let
$h\in \Hom_L(\hat{P}_{k+n-1},\hat{P}_k)$ and $k\neq -n$. Then $\partial^* h=h\circ \partial$ is given by
\begin{align*}
(\beta)&\stackrel{\partial}{\mapsto} \sum_i \pm z_i^{[\beta_i]} (\beta-\varepsilon^i) \\
 &\stackrel{h}{\mapsto} \sum_i \sum_\alpha \pm z_i^{[\beta_i]} h_{\alpha,\beta-\varepsilon^i} (\alpha). 
\end{align*}
(Here and in the following we write $\pm$ for a sign depending on the variables, which is not relevant for the proof.)
If $h\in C_k$, then $h_{\alpha,\beta-\varepsilon^i}\in \left<\tilde{\mathcal{I}}_{\alpha,\beta-\varepsilon^i}\right>_k$ and therefore
\[ \sum_i \pm z_i^{[\beta_i]} h_{\alpha,\beta-\varepsilon^i} \in \left<\tilde{\mathcal{J}}_{\alpha,\beta}\right>_k \]
by (ii); hence $\partial^* h\in A_k$. Similarly we get the implication $h\in D_k \implies \partial^* h\in B_k$ from (i). In the remaining case $k=-n$, the composition
$h\circ\partial$ is given by
\begin{align*}
(\svec{0})&\stackrel{\partial}{\mapsto} z_1^{[0]}z_2^{[0]}\dots z_r^{[0]} (\svec{-1}) \\
 &\stackrel{h}{\mapsto} \sum_\alpha \pm z_1^{[0]}z_2^{[0]}\dots z_r^{[0]} h_{\alpha,(\svec{-1})} (\alpha). 
\end{align*}
Together with (iii) one sees that $\partial^*$ maps $C_k$ to $0$. 

The proof for $\partial_*$ is done analogously using the corresponding claims
\begin{itemize}
\item[(I)] $z_i^{[\alpha_i+1]} \mathcal{I}_{\alpha+\varepsilon^i,\beta} 
              \subset \mathcal{J}_{\alpha,\beta}$,
\item[(II)] $z_i^{[\alpha_i+1]} \tilde{\mathcal{I}}_{\alpha+\varepsilon^i,\beta} 
               \subset \tilde{\mathcal{J}}_{\alpha,\beta}\cup \{0\}$,
\item[(III)] $z_1^{[0]}z_2^{[0]}\dots z_r^{[0]} \mathcal{\tilde{I}}_{\alpha,\beta}=0$.
\end{itemize}
We will prove (II) only; the proofs of (I) and (III) are similar to those of (i) and (iii) above. Again, it is enough to show the implication
$z_i^{[\alpha_i+1]}x\in\mathcal{J}_{\alpha,\beta} \implies x\in\mathcal{I}_{\alpha+\varepsilon^i,\beta}$ for every monomial $x$.
Let
\[ z_i^{[\alpha_i+1]}x=z_j^\nu \prod_{s\in S} z_s^{[\alpha_s+1]-[\beta_s+1]} \]
for some index set $S$, some $j$ not belonging to $S$, and $\nu\in\{[\alpha_j+1],[\beta_j]\}$. From $[\alpha_i+1]-[\beta_i+1]<[\alpha_i+1]$ we get $i\notin S$, and hence $j=i$. If $\nu=[\alpha_i+1]$, then we have
\[ x=\prod_{s\in S} z_s^{[\alpha_s+1]-[\beta_s+1]}\in\mathcal{I}_{\alpha+\varepsilon^i,\beta}. \]
On the other hand, if $\nu=[\beta_i]$ then in particular $[\beta_i]\geq [\alpha_i+1]$ and
\[ x=z_i^{[\beta_i]-[\alpha_i+1]}\prod_{s\in S} z_s^{[\alpha_s+1]-[\beta_s+1]}. \]
From the equality
\[ [\beta_i]-[\alpha_i+1]=[\alpha_i+2]-[\beta_i+1] \]
we get $x\in\mathcal{I}_{\alpha+\varepsilon^i,\beta}$. \qed

Let us now prove a generalised version of Proposition~\ref{llemma}.
\begin{lemma} \label{llemma2}
Let $k\in\mathbb{Z}$.
\begin{itemize}
\item[(i)] Let $f:\hat{P}_{n+k}\longrightarrow \hat{P}_k$ be an element of $A_k$ satisfying $\partial\circ f=0$. If $k\geq 0$, then there is some $h:\hat{P}_{n+k}\longrightarrow \hat{P}_{k+1}$ in $C_{k+1}$ such that $f=\partial\circ h$.
\item[(ii)] Let $f:\hat{P}_{n+k}\longrightarrow \hat{P}_k$ be an element of $A_k$ satisfying $f\circ \partial=0$. If $n+k<0$, then there is some $h:\hat{P}_{n+k-1}\longrightarrow \hat{P}_k$ in $C_k$ such that $f=h\circ\partial$.
\end{itemize}
\begin{align*} \xymatrix{
 & \hat{P}_{n+k}\ar[d]^{f}\ar@{-->}[dr]^{h} \\
 \hat{P}_{k-1} & \hat{P}_k\ar[l]^-\partial & \hat{P}_{k+1}\ar[l]^-\partial }  
&&
 \xymatrix{
 \hat{P}_{n+k-1}\ar@{-->}[dr]_h & \hat{P}_{n+k}\ar[l]^-\partial\ar[d]_f & \hat{P}_{n+k+1}\ar[l]^\partial \\
  & \hat{P}_k }
  \end{align*}
\end{lemma}
\begin{proof}
We will prove (i) only. Since $\hat{P}_{n+k}$ is projective and $\partial\circ f=0$, there is some map $h'\in\Hom_L(\hat{P}_{n+k},\hat{P}_{k+1})$ satisfying $f=\partial\circ h'=\partial_* h'$. Denote by $h$ the image of $h'$ under the retraction
$\Hom_L(\hat{P}_{n+k},\hat{P}_{k+1})=C_{k+1}\oplus D_{k+1} \longrightarrow C_{k+1}$.
By Proposition~\ref{partialzerlegung} one has the commutative diagram
\begin{align*}
\xymatrix{
C_{k+1}\oplus D_{k+1} \ar[r]\ar[d]^{\partial_*} & C_{k+1} \ar[d]^{\partial_*\vert_{C_{k+1}}} \\
A_k\oplus B_k \ar[r] & A_k } 
&&
\xymatrix{ h' \ar@{|->}[r] \ar@{|->}[d] & h \ar@{|->}[d] \\ f \ar@{|->}[r] &  f } 
\end{align*}
This proves the proposition.
\end{proof}

\begin{proof}[Proof of Theorem~\ref{flemma2} (cont.)]
Let $f:\hat{P}[n]\longrightarrow \hat{P}$ of non-positive degree $n$ be a cocycle and, at the same time, a $J_2$-map. We want to construct a null-homotopy for $h$ which is an $I$-map. Note that all the matrix entries of $f$ consist of sums of monomials of the form $z_i^{m_i-1}z_j^{m_j-1}z$ (where $i\neq j$, and $z$ is some monomial in the other variables), which belong to $\tilde{\mathcal{J}}_{\alpha,\beta}$. Thus $f_k:\hat{P}_{k+n}\longrightarrow \hat{P}_k$ is an element of $A_k$ for all $k$.

Now we will construct $h$ inductively in such a way that $h_k\in C_k$ for all $k$.
Put $h_0=0$. We know that $\partial\circ f_0=0$, because $\partial=\partial_0$ is given by multiplication with $\norm=\prod_i z_i^{m_i-1}$ and $\norm\cdot J_2\subset \norm\cdot I=0$. From
$f_0\in A_0$ and Proposition~\ref{llemma2}.(i) we get some $h_1:\hat{P}_n\longrightarrow \hat{P}_1$ satisfying $f_0=\partial\circ h_1$ and $h_1\in C_1$. 

Suppose we have constructed $h_0,h_1,h_2,\dots,h_k$ in such a way that $h_i\in C_i$ for all $i$ and $f_i=\partial h_{i+1}+(-1)^{n}h_i\partial$ for all $i=0,1,\dots,k-1$. Consider $f_k'=f_k-(-1)^n h_k\circ \partial$.
\[
\xymatrix@R=40pt@C=21pt{
\dots & \hat{P}_{n-1}\ar[l]\ar[d]_{f_{-1}}\ar@{-->}[dr]_{0} & \hat{P}_{n}\ar[l]\ar[d]_{f_{0}}\ar@{-->}[dr]_{h_1} &
\hat{P}_{n+1}\ar[l]\ar[d]_{f_1} & \dots\ar[l] & \hat{P}_{n+k-1}\ar[l]\ar[d]_{f_{k-1}}\ar@{-->}[dr]_{h_{k}} &
\hat{P}_{n+k}\ar[l]\ar[d]_{f_{k}}\ar@{-->}[dr]_{h_{k+1}} & \hat{P}_{n+k+1}\ar[l]\ar[d]_{f_{k+1}} & \dots\ar[l] \\
\dots & \hat{P}_{-1}\ar[l] & \hat{P}_{0}\ar[l] &
\hat{P}_{1}\ar[l] & \dots\ar[l] & \hat{P}_{k-1}\ar[l] &
\hat{P}_{k}\ar[l] & \hat{P}_{k+1}\ar[l] & \dots\ar[l] \\
}
\]
Since $h_k\in C_k$, we get $h_k\circ\partial\in A_k$ by Proposition~\ref{partialzerlegung}, and hence $f_k'\in A_k$. By induction hypothesis we have 
\[ \partial f_k'=\partial f_k-(-1)^n \partial h_k\partial=(-1)^n (f_{k-1}-\partial h_k)\partial
 =h_{k-1}\partial\partial =0. \]
By Proposition~\ref{llemma2}.(i) we get a map $h_{k+1}\in C_{k+1}$ satisfying $f_k'=\partial h_{k+1}$, i.e.~$f_k=\partial h_{k+1}+(-1)^n h_k\partial$. Using the same induction argument in negative degrees we end up with a homotopy $h$ satisfying $h_k\in C_k$ for all integers $k$. Since all elements of $C_k$ are $I$-maps, we are done.
\end{proof}

\subsection{Proof of the main theorem}
\begin{lemma} We have
\[ \Tate^*_{kG}(k,k)\cong k[u_i,v_i,\varphi_\alpha] / \sim \]
where $i=1,2,\dots,r$ and $\alpha$ runs through all multi-indices in $\mathbb{Z}_{\geq 0}^r$. The brackets on the right hand side
mean \emph{graded} commutativity. The relations are given by
\begin{eqnarray*}
u_i^2&=&\begin{cases} v_i & \text{if $m_i=2$} \\
 0 & \text{otherwise} \end{cases} \\
\varphi_\alpha u_i&=&\begin{cases} (-1)^{\alpha_i+\dots+\alpha_r+i}\,\varphi_{\alpha-\varepsilon^i} & \text{if $\alpha_i$ is odd, or $\alpha_i>0$ is even and $m_i=2$} \\ 0 & \text{otherwise} \end{cases} \\
\varphi_\alpha v_i&=&\begin{cases} \varphi_{\alpha-2\varepsilon^i} & \text{if $\alpha_i\geq 2$} \\ 0 & \text{otherwise} \end{cases} \\
\varphi_\alpha \varphi_\beta &=& 0  
\end{eqnarray*}
for all $i$ and all multi-indices $\alpha$ and $\beta$. The degrees are $|u_i|=1$, $|v_i|=2$, and $|\varphi_\alpha|=-|\alpha|-1$. 
\end{lemma}
\beweis
First, let us restrict to the non-negative case. The claim is
\begin{equation} \label{kunneth1} \Ext^*_{kG}(k,k)\cong k[\{u_i,v_i\}_{i=1}^r] / \sim \end{equation}
(where $\sim$ stands for the first of the relations given in the Proposition). The claim follows immediately from the Künneth isomorphism, because both sides are the tensor product of all $H^*(G_i,k)\cong k[u_i,v_i]/\sim$ for $i=1,2,\dots,r$. But we are interested in explicit representatives of the $u_i$'s and $v_i$'s in the endomorphism algebra. Note that (see \cite{bks}, Example~7.7) the Künneth isomorphism is induced by a map of endomorphism-dga's as follows. By neglecting the negative part of the complete resolutions $\hat{P}^i_*$ and $\hat{P}_*$ we get 'ordinary' projective resolutions $P^i_*$ and $P_*$ of $k$. By construction, $P_*=\bigotimes_i P^i_*$. One has a quasi-isomorphism of dga's
\begin{align} \label{kunnethpre} \bigotimes_{i=1}^r \Hom_{kG_i}^*(P^i_*,P^i_*) \longrightarrow \Hom_{kG}^*(P_*,P_*), 
\end{align}
given by tensoring endomorphisms, which induces the Künneth isomorphism. Let us denote by $\Bar{x}'_i$ the non-negative part of $\Bar{x}_i$; this is a representative of $x_i$ considered as an element of $H^*(G_i,k)$. By construction of $\Phi$, the morphism \eqref{kunnethpre} maps
\[ \Id\otimes \Id\otimes\dots\otimes\Bar{x}'_i\otimes\dots\otimes\Id \]
to the non-negative part of $\Phi(\Bar{x}_i)$. Similarly, this holds for $\Bar{y}_i$. 
Therefore, if we define
\[ \Bar{u}_i=\Phi(\Bar{x}_i)\quad\text{and}\quad \Bar{v}_i=\Phi(\Bar{y}_i), \]
we get representatives for the $u_i$'s and $v_i$'s in (\ref{kunneth1}).

Now we are going to construct the $\varphi_\alpha$'s. Fix some multi-index $\alpha$. For every $i$ we can write $ -\alpha_i-1=2\beta_i+\epsilon_i$, where $\beta_i\in\mathbb{Z}_-$ and $\epsilon_i\in\{0,1\}$. Now consider $ \Bar{f}_i=\Bar{y}_i^{\beta_i} \Bar{x}_i^{\epsilon_i} $.
It is easy to check that the map in degree zero
\[ \bigl(\Bar{f}_i\bigr)_0:\hat{P}^i_{-\alpha_i-1}\longrightarrow \hat{P}^i_0 \]
is the identity map. Let us define  $\Bar{\varphi}_\alpha = \Psi(\bar{f}_1,\bar{f}_2,\dots,\bar{f}_r)$. The degree of this map is $|\bar{\varphi}_\alpha|=-|\alpha|-1$. By construction of $\Psi$, the map in degree $0$ 
\[ \bigl(\Bar{\varphi}_\alpha\bigr)_0:\hat{P}_{-|\alpha|-1}\longrightarrow \hat{P}_0 \]
is (up to a sign) the projection onto the summand $\bigotimes_i \hat{P}^i_{-\alpha_i-1} \subseteq \hat{P}_{-|\alpha|-1}$. This implies that under the isomorphism $\smash{\Tate}^{-|\alpha|-1}_{kG}(k,k)\cong k^{M_{|\alpha|}}$ the class of this map is sent to the $k$-basis element corresponding to $\alpha$. In particular, the classes $\varphi_\alpha$ of the $\bar{\varphi}_\alpha$ are $k$-linearly independent. 

The relations given in the Proposition will follow from the following
\begin{lemma} \label{kozrelationen1} For all multi-indices $\alpha,\beta$ and all $j$ we have
\begin{align}
 (-1)^{\alpha_1+\dots+\alpha_{j-1}+j-1}\bar{u}_j \bar{\varphi}_\alpha \simeq
 (-1)^{\alpha_j+\dots+\alpha_r+j}\bar{\varphi}_\alpha \bar{u}_j &\simeq\begin{cases}
     \bar{\varphi}_{\alpha-\varepsilon^j}  & \text{if $\alpha_j$ is odd, or $\alpha_j>0$ and $p_j=1$} 
     \\ 0 & \text{otherwise}  \end{cases}    \label{relationen1} \\
 \bar{v}_j \bar{\varphi}_\alpha \simeq\bar{\varphi}_\alpha \bar{v}_j &\simeq\begin{cases}
     \bar{\varphi}_{\alpha-2\varepsilon^j}  & \text{if $\alpha_j\geq 2$} 
     \\ 0 & \text{otherwise}  \end{cases}    \label{relationen2} \\
 \bar{\varphi}_\alpha\bar{\varphi}_\beta &\simeq 0 \label{relationen3}
\end{align}
If $r\geq 3$ and $m_i\neq 3$ for all $i$, then all these homotopies can be chosen to be $I$-maps.
\end{lemma}
\beweis
Ad \eqref{relationen1}:
We will use the notation of the previous proof. Due to Proposition~\ref{psiphikombo} we have modulo $J_{r-1}$-maps
\begin{align*}
 (-1)^{\alpha_j+\dots+\alpha_r+j} \bar{\varphi}_\alpha \bar{u}_j &= (-1)^{\alpha_j+\dots+\alpha_r+j} \Psi(\bar{f}_1,\dots,\bar{f}_r)\Phi(\Bar{x}_j)  \\
 &\equiv \begin{cases} \Psi(\bar{f}_1,\dots,\bar{f}_j \bar{x}_j,\dots,\bar{f}_r) & \text{if $|\bar{f}_j \bar{x}_j|<0$} \\ 0 & \text{otherwise} \end{cases}  \\
 &=\begin{cases} \Psi(\bar{f}_1,\dots,\bar{y}_j^{\beta_j}\bar{x}_j^{\epsilon_j+1},\dots,\bar{f}_r) & \text{if $\alpha_j>0$} \\ 0 &\text{otherwise} \end{cases}
\intertext{Similarly, from $\bar{x}_j\bar{y}_j=\bar{y}_j\bar{x}_j$ we get}
 (-1)^{\alpha_1+\dots+\alpha_{j-1}+j-1} \bar{u}_j \bar{\varphi}_\alpha &\equiv \begin{cases} \Psi(\bar{f}_1,\dots,\bar{x}_j\bar{f}_j,\dots,\bar{f}_r) & \text{if $| \bar{x}_j\bar{f}_j|<0$} \\ 0 & \text{otherwise} \end{cases}  \\
 &=\begin{cases} \Psi(\bar{f}_1,\dots,\bar{y}_j^{\beta_j}\bar{x}_j^{\epsilon_j+1},\dots,\bar{f}_r) & \text{if $\alpha_j>0$} \\ 0 &\text{otherwise} \end{cases}
\end{align*}
If $\alpha_j$ is odd, then $\epsilon_j=0$, and by definition
\[ \Psi(\bar{f}_1,\dots,\bar{y}_j^{\beta_j}\bar{x}_j^{\epsilon_j+1},\dots,\bar{f}_r)
 =\bar{\varphi}_{\alpha-\varepsilon^j}. \]
If $\alpha_j$ is positive and even, then $\bar{x}_j^{\epsilon_j+1}=\bar{x}_j^2$, and we are left with two possible cases: If $m_j=2$, then $\bar{x}_j^2=\bar{y}_j$ and 
\[ \Psi(\bar{f}_1,\dots,\bar{y}_j^{\beta_j}\bar{x}_j^{\epsilon_j+1},\dots,\bar{f}_r)
=\Psi(\bar{f}_1,\dots,\bar{y}_j^{\beta_j+1},\dots,\bar{f}_r)
 =\bar{\varphi}_{\alpha-\varepsilon^j}. \]
If $m_j\geq 3$, then $\bar{x}_j^2$ is null-homotopic via $\bar{q}_j$. By Proposition~\ref{dpsikomm}, $\Psi(\bar{f}_1,\dots,\bar{y}_j^{\beta_j}\bar{x}_j^{\epsilon_j+1},\dots,\bar{f}_r)$ is null-homotopic via $\Psi(\bar{f}_1,\dots,\bar{y}_j^{\beta_j}\bar{q}_j,\dots,\bar{f}_r)$. This is an $I$-map, because $\bar{q}_j$ is a $\bigl<x_j^{m_j-3}\bigr>_{kG_j}$-map and $x_j^{m_j-3}\in I$ (due to $m_j\geq 4$). 

This can be summarized as follows: The cocycles we considered are equal up to a $J_{r-1}$-map and possibly some homotopy which is an $I$-map if $m_j\neq 3$. Since $r\geq 2$, we know that $J_{r-1}\subseteq I$. By Proposition~\ref{nulllemma}, the cocycles represent the same cohomology class; therefore, they are homotopic. If $r\geq 3$ and $m_j\neq 3$, the homotopy can be chosen to be an $I$-map by Theorem~\ref{flemma2}.\bigskip

Ad \eqref{relationen2}:
Since $\bar{x}_j$ and $\bar{y}_j$ commute, we know by Proposition~\ref{psiphikombo} modulo $J_{r-1}$-maps
\begin{eqnarray*}
 \bar{\varphi}_\alpha \bar{v}_j = \Psi(\bar{f}_1,\dots,\bar{f}_r)\Phi(\Bar{y}_j)
 &\equiv& \begin{cases} \Psi(\bar{f}_1,\dots,\bar{f}_j \bar{y}_j,\dots,\bar{f}_r) & \text{if $|\bar{f}_j \bar{y}_j|<0$} \\ 0 & \text{otherwise} \end{cases}  \\
 &=&\begin{cases} \Psi(\bar{f}_1,\dots,\bar{y}_j^{\beta_j+1}\bar{x}_j^{\epsilon_j},\dots,\bar{f}_r) & \text{if $\alpha_j>1$} \\ 0 &\text{otherwise} \end{cases} \\
 &=&\begin{cases} \bar{\varphi}_{\alpha-2\varepsilon^j} & \text{if $\alpha_j>1$} \\ 0 & \text{otherwise} \end{cases}
\end{eqnarray*}
and similarly for $\bar{v}_j\bar{\varphi}_\alpha$.\bigskip

The relation (\ref{relationen3}) follows directly from Proposition~\ref{phikombo}.\qed

\begin{lemma} \label{kozrelationen2} If $r\geq 3$, then:
\begin{itemize}
\item[(a)] Every two cocycles from $\{\bar{v}_1,\dots,\bar{v}_r,\bar{u}_1,\dots,\bar{u}_r\}$ commute (in the graded sense) up to a homotopy which is an $I$-map. 
\item[(b)] If $m_i\geq 3$, then $\bar{u}_i^2$ is null-homotopic. If $m_i\geq 4$, then the homotopy can be chosen to be an $I$-map. 
\end{itemize}
\end{lemma}
\beweis
Ad (a): If $j=i$, then $\bar{v}_j\bar{u}_j=\bar{u}_j\bar{v}_j$, since $\bar{x}_j\bar{y}_j=\bar{y}_j\bar{x}_j$. Assume $i\neq j$. For arbitrary cocycles $f\in A^\vdx{i}$, $g\in A^\vdx{j}$ of positive degree we know by Remark~\ref{phifbemerkung} and Proposition~\ref{phifphigkomm} that $\Phi(f)\Phi(g)$ and $(-1)^{ij}\Phi(g)\Phi(f)$ are homotopic via some $J_{r-2}$-map. All the $\bar{v}_j$ and $\bar{u}_j$ are of the form $\Phi(f)$, and $J_{r-2}$-maps are $I$-maps, because $r\geq 3$. This proves (a).

Ad (b): By Proposition~\ref{phidkomm} $\bar{u}_i^2$ is null-homotopic via $\Phi(\bar{q}_i)$, which is an $I$-map if $m_i\geq 4$.\qed

As an immediate consequence of the previous two Propositions we get
\begin{korollar} \label{allekozrelationen} Suppose $r\geq 3$ and $m_i\neq 3$ for all $i$. Let $\bar{f}$ and $\bar{g}$ be compositions of cocycles of the set
$\{\bar{u}_i,\bar{v}_i,\bar{\varphi}_\alpha\}_{i,\alpha},$
and assume that one can get $\bar{g}$ from $\bar{f}$ using the relations \eqref{relationen1}, \eqref{relationen2}, \eqref{relationen3}, graded commutativity and
\[ \bar{u}_i^2 \simeq \begin{cases} \bar{v}_i & \text{if $m_i=2$} \\ 0 & \text{otherwise.} \end{cases} \]
Then $\bar{f}$ and $\bar{g}$ are homotopic via an $I$-map.
\end{korollar}

Now we can prove our main result.
\begin{satz} \label{realisierbar}
If $r\geq 3$ and $m_i\neq 3$ for all $i$, then $\gamma_G=0$.
\end{satz}
\begin{proof}
Let us construct a representative $m$ for $\gamma_G$. We need to choose a $k$-linear map
\[ f_1:H^*A\longrightarrow A \]
of degree $0$, such that $a\in H^*A$ is mapped to a cocycle $f_1(a)$ in $A$ representing $a$. We do this as follows: For all multi-indices $\beta,\epsilon$ with $\epsilon_i\in\{0,1\}$ (for all $i$) put
\begin{eqnarray*}
 H^*A&\stackrel{f_1}{\longrightarrow}& A \\
 u_1^{\epsilon_1}v_1^{\beta_1}u_2^{\epsilon_2}v_2^{\beta_2}\dots u_r^{\epsilon_r}v_r^{\beta_r}
 & \mapsto & \bar{u}_1^{\epsilon_1}\bar{v}_1^{\beta_1}\bar{u}_2^{\epsilon_2}\bar{v}_2^{\beta_2}\dots \bar{u}_r^{\epsilon_r}\bar{v}_r^{\beta_r} 
\end{eqnarray*}
and for all multi-indices $\alpha$
\begin{eqnarray*}
 H^*A&\stackrel{f_1}{\longrightarrow}& A \\
 \varphi_\alpha
 & \mapsto & \bar{\varphi}_\alpha.
\end{eqnarray*}
Now we are looking for a $k$-linear map 
\[ f_2:H^*A\longrightarrow A \]
of degree $-1$ with the property $df_2(x,y)=f_1(xy)-f_1(x)f_1(y)$ for all homogeneous $x,y\in H^*A$. By construction of $f_1$ we can apply Corollary~\ref{allekozrelationen} to the cocycles $f_1(xy)$ and $f_1(x)f_1(y)$. This guarantees the existence of a suitable $f_2(x,y)$ which is an $I$-map. From Proposition~\ref{ilemma} we get $m=0$, and hence $\gamma_G=0$.
\end{proof}

Now we are going to investigate the case $r=2$.
\begin{satz} \label{nichtrealisierbar}
Let $r=2$. Then we have
\begin{align}\label{masseyklammer}
 \left< v_2,\varphi_{(0,1)},v_1 \right> = u_1 \in\Tate^1_{kG}(k,k) 
\end{align}
without indeterminacy. In particular $\gamma_G\neq 0$, and the $\hat{H}^*(G,k)$-module
\[ X=\hat{H}^*(G,k)/v_2\hat{H}^*(G,k) \]
is (by Proposition~\ref{cokernlemma}) not a direct summand of a realisable module.
\end{satz}
\begin{proof}
Since $v_2\,\varphi_{(0,1)}=\varphi_{(0,1)}v_1=0$, the Massey product is defined. The indeterminacy is given by $v_2\,\hat{H}^{-1}(G)+\hat{H}^{-1}(G)\,v_1=0$. 
To prove equation \eqref{masseyklammer} we choose the cocycles $\bar{v}_1,\bar{v}_2$ and $\bar{\varphi}_{(0,1)}$ as representatives for $v_1,v_2$ and $\varphi_{(0,1)}$, respectively. By definition of the Massey product, we have to choose null-homotopies for $\bar{v}_2\bar{\varphi}_{(0,1)}$ and $\bar{\varphi}_{(0,1)}\bar{v}_1$. Consider $\bar{\varphi}_{(0,1)}\bar{v}_1$: In degree $1$ this (i.e.~the map $\hat{P}_1\longrightarrow \hat{P}_{-1}\longrightarrow \hat{P}_1$) is given by
\begin{eqnarray*}
a^1\otimes a^2&\stackrel{\bar{v}_1}{\longmapsto}& \begin{cases} \bar{y}_1(a^1)\otimes\partial(a^2) & 
      \text{if $|a^1|=1, |a^2|=0$} \\         0 & \text{otherwise} \end{cases} \\
              &\stackrel{\bar{\varphi}_{(0,1)}}{\longmapsto}& \begin{cases} -\bar{y}_1^{-1}\bar{x}_1\bar{y}_1(a^1)\otimes \bar{y}_2^{-1}\partial (a^2) & \text{if $|a^1|=1, |a^2|=0$}  \\         0 & \text{otherwise} \end{cases} 
\end{eqnarray*}

\begin{minipage}[t][0cm]{11.5cm}
We can lift this to a map $\hat{P}_1\longrightarrow \hat{P}_2$ as follows: Define $\bar{\sigma}_1$ as
\[ a^1\otimes a^2\mapsto \begin{cases} -\bar{x}_1(a^1)\otimes \bar{y}_2^{-1}(a^2) & \text{if $|a^1|=1, |a^2|=0$} \\ 0 & \text{otherwise} \end{cases}
\]
\end{minipage}
$\xymatrix{
\hat{P}_1 \ar[d]_{\bar{\varphi}_{(0,1)}\bar{v}_1}\ar[dr]^{\bar{\sigma}_1} \\
\hat{P}_1 & \hat{P}_2\ar[l]^{\partial} } $

Since $\partial(-\bar{x}_1(a^1)\otimes \bar{y}_2^{-1}(a^2))=-\bar{x}_1(a^1)\otimes \partial\bar{y}_2^{-1}(a^2)$ we see that the diagram to the left commutes. 
In particular we can extend $\bar{\sigma}_1$ to a null-homotopy $\bar{\sigma}$ of $\bar{\varphi}_{(0,1)}\bar{v}_1$.
Finally, choose any homotopy $\bar{\tau}$ for $\bar{v}_2\bar{\varphi}_{(0,1)}$. 

Now we need to determine the class of the cocycle
$\bar{\tau}\bar{v}_1-\bar{v}_2\bar{\sigma}$.
Since
\[ \class(\bar{\tau}\bar{v}_1-\bar{v}_2\bar{\sigma})=\class(\bar{\tau}\bar{v}_1)-\class(\bar{v}_2\bar{\sigma}), \]
we can consider the summands separately. The map in degree $0$ determines the class uniquely; we therefore have a look at
$\hat{P}_1 \stackrel{\bar{v}_1}{\longrightarrow} \hat{P}_{-1} \stackrel{\bar{\tau}}{\longrightarrow} \hat{P}_0.$
By definition, $\bar{v}_1$ is a $J_1$-map. Therefore, it is also an $I$-map, and $\class(\bar{\tau}\bar{v}_1)=0$. Next consider
$\hat{P}_1 \stackrel{\bar{\sigma}_1}{\longrightarrow} \hat{P}_{2} \stackrel{\bar{v}_2}{\longrightarrow} \hat{P}_0$.
This composition is given by
\begin{align*}
a^1\otimes a^2&\stackrel{\bar{\sigma}_1}{\longmapsto} \begin{cases} -\bar{x}_1(a^1)\otimes \bar{y}_2^{-1}(a^2) & \text{if $|a^1|=1, |a^2|=0$} \\ 0 & \text{otherwise} \end{cases} \\
&\stackrel{\bar{v}_2}{\longmapsto} \begin{cases} 
-\bar{x}_1(a^1)\otimes a^2 & \text{if $|a^1|=1, |a^2|=0$} \\ 0 & \text{otherwise.} \end{cases}
\end{align*}
This is the same as $-\bar{u}_1:\hat{P}_1\longrightarrow \hat{P}_0$. Altogether we have
\[ \class(\bar{\tau}\bar{v}_1-\bar{v}_2\bar{\sigma})=\class(\bar{\tau}\bar{v}_1)-\class(\bar{v}_2\bar{\sigma})=0-\class(-\bar{u}_1)=u_1, \]
which had to be shown.
\end{proof}

\bemerkung
The proof of Theorem~\ref{nichtrealisierbar} might become somewhat clearer as soon as the maps used are given as matrices. We have the isomorphism
\[ \hat{P}_n=L^{M_n}\cong L^{n+1}. \]
Here we identify the summand on the left hand side corresponding to a multi-index $(a,b)$ with the summand on the right hand side corresponding to the index $a+1$. The map $\bar{\varphi}_{(0,1)}\,\bar{v}_1$ is given in degree $1$ by
\newcommand{\vla}[2]{\stackrel{#1}{\xrightarrow{\mspace{#2mu}}}}
\[
\begin{array}{ccccc}
 \hat{P}_1 &\vla{\bar{v}_1}{90}& \hat{P}_{-1} &\vla{\bar{\varphi}_{(0,1)}}{70} & \hat{P}_1 \\[6pt]
 \parallel & & \parallel & & \parallel \\
 L^2 &\vla{\smatrix{z_2^{m_2-1} & 0}}{90} & L & \vla{\smatrix{0 \\ -1}}{70} & L^2 
\end{array}
\]
The composition can be written as the matrix $\smatrix{0 & 0 \\ -z_2^{m_2-1} & 0}$.
The lifting diagram occurring in the proof is of the form
\[ \xymatrix@R=30pt@C=50pt{
L^2\ar[d]_{\bigl(\begin{smallmatrix} 0 & 0 \\ -z_2^{m_2-1} & 0 \end{smallmatrix}\bigr)}
\ar[dr]^{\Bigl(\begin{smallmatrix} 0 & 0 \\ 0 & 0 \\ -1 & 0 \end{smallmatrix}\Bigr)} \\
L^2 & L^3 \ar[l]^{\partial=\Bigl(\begin{smallmatrix} z_1^{m_1-1} & -z_2 & 0 \\ 0 & z_1 & z_2^{m_2-1} \end{smallmatrix}\Bigr)} } \]
Composing the lift with the degree $0$-part of $\bar{v}_2$, i.e.~
\[ \bigl(\bar{v}_2\bigr)_0:L^3\longrightarrow L, \]
which is the projection onto the third factor, we get (up to a sign) the projection $L^2\rightarrow L$ onto the second factor, which in turn corresponds to $-\bar{u}_1$ in degree $0$.\\

Finally, let us consider the case $m_i=3$ for some $i$ (and $r$ is arbitrary). We may assume that $i=1$.
\begin{satz} \label{nichtrealisierbar3}
Suppose $m_1=3$. Then
\begin{align}\label{masseyklammer3}
 0 \not\in \left< u_1,u_1,u_1 \right>  \subseteq\Tate^2_{kG}(k,k) 
\end{align}
Hence, $\gamma_G\neq 0$ and the $\hat{H}^*(G,k)$-module
\[ X=\hat{H}^*(G,k)/u_1\hat{H}^*(G,k) \]
is (by Proposition~\ref{cokernlemma}) not a direct summand of a realisable module.
\end{satz}
\begin{proof}
The indeterminacy is given by $u_1 \Tate^1_{kG}(k,k) \not\ni v_1$. We will show that $v_1\in \left< u_1,u_1,u_1 \right>$; then the result follows.
We choose $\bar{u}_1$ as a representing cocycle for $u_1$. Then $\bar{u}_1^2=\Phi(\bar{x}_1^2)$ is the boundary of $\Phi(\bar{q}_1)$ by Proposition~\ref{phidkomm}. But then
\[ \Phi(\bar{q}_1)\Phi(\bar{x}_1)+\Phi(\bar{x}_1)\Phi(\bar{q}_1) = \Phi(\bar{q}_1\bar{x}_1+\bar{x}_1\bar{q}_1) = \Phi(\bar{y}_1) = \bar{v}_1, \]
and therefore $v_1\in \left< u_1,u_1,u_1 \right>$.
\end{proof}
Note that Theorem~\ref{rsatzallgemein} follows from Theorems~\ref{realisierbar}, ~\ref{nichtrealisierbar}, and \ref{nichtrealisierbar3}.


\begin{thebibliography}{BKS04}

\bibitem[AM04]{AdemMilgram}
Alejandro Adem and R.~James Milgram, \emph{Cohomology of finite groups}, second
  ed., Grundlehren der Mathematischen Wissenschaften [Fundamental Principles of
  Mathematical Sciences], vol. 309, Springer-Verlag, Berlin, 2004.
  \MR{MR2035696 (2004k:20109)}

\bibitem[BKS04]{bks}
David Benson, Henning Krause, and Stefan Schwede, \emph{Realizability of
  modules over {T}ate cohomology}, Trans. Amer. Math. Soc. \textbf{356} (2004),
  no.~9, 3621--3668 (electronic). \MR{MR2055748 (2005b:20102)}

\bibitem[Car96]{Carlson}
Jon~F. Carlson, \emph{Modules and group algebras}, Lectures in Mathematics ETH
  Z\"urich, Birkh\"auser Verlag, Basel, 1996, Notes by Ruedi Suter.
  \MR{MR1393196 (97c:20013)}

\bibitem[CE99]{CartEile}
Henri Cartan and Samuel Eilenberg, \emph{Homological algebra}, Princeton
  Landmarks in Mathematics, Princeton University Press, Princeton, NJ, 1999,
  With an appendix by David A. Buchsbaum, Reprint of the 1956 original.
  \MR{MR1731415 (2000h:18022)}

\bibitem[May69]{May}
J.~Peter May, \emph{Matric {M}assey products}, J. Algebra \textbf{12} (1969),
  533--568. \MR{MR0238929 (39 \#289)}

\end{thebibliography}
\end{document}